\documentclass[11pt,twoside, a4paper, english, reqno]{amsart}
 \usepackage[dvips]{epsfig}
 \usepackage{amscd}
 \usepackage{amssymb}
 \usepackage{amsthm}
 \usepackage{amsmath}
 \usepackage{bm} 
 \usepackage{latexsym}
 
 \usepackage{upref}
 \usepackage[hidelinks]{hyperref}
 \usepackage{color}
 \usepackage[usenames,dvipsnames]{xcolor}
 \usepackage{enumitem}
 \usepackage{comment}
\usepackage{todonotes}
 \usepackage{ulem}
 \usepackage{float}
\usepackage{tikz}
 
 \theoremstyle{plain}

\renewcommand{\O}{\Omega}

\newcommand{\bel}[1]{\begin{equation}\label{#1}}
\newcommand{\ee}{\end{equation}}

\newtheorem{theorem}{Theorem}[section]
\newtheorem{corollary}[theorem]{Corollary}

\newtheorem{lemma}[theorem]{Lemma}
\newtheorem{proposition}[theorem]{Proposition}

\newtheorem{case}[theorem]{Case}
\theoremstyle{definition}
\newtheorem{definition}[theorem]{Definition}

\title[Inverse obstacle scattering problem]{Uniqueness in inverse elastic scattering from traction-free polygonal obstacles with a single incoming wave}
\author[Hu]{Guanghui Hu}
\address{School of Mathematical Sciences and LPMC, Nankai University (NKU), Tianjin, 300071, China}
\email{ghhu@nankai.edu.cn
 }
\author[Rana]{Vinay Rana}
\address{ 
Department of Mathematics\\
Indian Institute of Technology, Ropar, Rupnagar, Punjab, 140001, India}
\email{vinay.23maz0011@iitrpr.ac.in}

\author[Vashisth]{Manmohan Vashisth}
\address{Department of Mathematics\\
Indian Institute of Technology, Ropar, Rupnagar, Punjab, 140001, India}
\email{manmohanvashisth@iitrpr.ac.in}

\begin{document}
\begin{abstract} 
In this article, we establish uniqueness in the inverse problem of time-harmonic elastic scattering from a bounded, traction-free obstacle in two dimensions. For a single incident plane wave, it is proved that a connected polygonal obstacle is uniquely determined by the measured far-field pattern. The analysis is based on a reflection principle for the Navier equation under traction-free boundary conditions and unique continuation of elastic fields.
\end{abstract}

\maketitle
\vspace{-0.8cm}
\section{Introduction and statement of  Main Result}
\subsection{Mathematical formulations and main result}\label{mathematical formulation}
We begin with a set of notation and definitions that will be used throughout this article. For  a vector $x:=(x_1,x_2) \in \mathbb{R}^2$, we introduce two unit vectors $\widehat{x}:=x/|x|$ and $\widehat{x}^\perp$, the vector obtained by rotating 
$\widehat{x}$ in the counterclockwise direction  by an angle of $\pi/2$. For a vector field $u=(u_1,u_2)$, defined in an open subset of $\mathbb{R}^2$, we use the following notation to represent various differential operators
\begin{align*}
    \begin{aligned}
    \nabla\cdot u:=\partial_1 u_1+\partial_2 u_2, \ \nabla\times u:= \partial_2u_1 - \partial_1u_2,\ 
\nabla u_j :=\begin{bmatrix}
    \partial_1u_j\\ \partial_2u_j
\end{bmatrix} \ \mbox{and} \  \nabla^{\perp}u_j:=\begin{bmatrix}
    -\partial_2 u_j\\
    \partial_1 u_j
\end{bmatrix},
    \end{aligned} 
\end{align*}
where $j=1,2$ and 
$\partial_i$ stands for the first order partial derivative in the direction of $x_i$, for $i=1,2.$ Throughout, the cross-product operator $\nabla\times$ is understood in the two-dimensional sense: applied to a vector field $u=(u_1,u_2)$ it yields the scalar $\nabla\times u=\partial_2u_1-\partial_1u_2$, whereas applied to a scalar function $\varphi$ it yields the vector field $\nabla\times\varphi:=(-\partial_2\varphi,\partial_1\varphi)=\nabla^{\perp}\varphi$. In particular, the identity $\nabla\times(\nabla\times u)=\nabla(\nabla\cdot u)-\Delta u$ holds. We  define the Lam\'e operator $\mathcal{L}$ acting on vector fields in $\mathbb{R}^2$ by \begin{align}\label{Lame oeprator}
    \mathcal{L}u := \mu \Delta u + (\lambda + \mu)\, \nabla( \nabla \cdot u),  
\end{align}
where $\lambda$ and $\mu $ denote the 
Lam\'e constants such that $\mu>0$ and $\lambda+2\mu>0$. In addition, we require $\lambda+\mu>0$, so that the constant $c:=\mu/(\lambda+\mu)$ introduced below is well defined and positive. We also assume that the density of the background medium equals one, i.e., $\rho\equiv 1$. Now, for a fixed frequency $\omega>0$, we denote by $\mathcal{N}_{\omega}$ the Navier operator, defined by \begin{align}\label{Navier operator}
\mathcal{N}_{\omega}u:=\mathcal{L}u+\omega^2 u
\end{align} 
where $\mathcal{L}$ represents the Lam\'e operator as given by equation \eqref{Lame oeprator}. Throughout this article, we assume that the incident field, denoted by $u^{in}$, is a vector field in $\mathbb{R}^2$ and satisfies the following free space Navier equation 
\begin{align}\label{Equation for Incident field} 
\mathcal{N}_{\omega} u^{in}(x):=\mu \Delta u^{in}(x)+ (\lambda + \mu)\, \nabla( \nabla \cdot u^{in})(x) + \omega^2 u^{in}(x)=0,\ x\in \mathbb{R}^2. 
\end{align} 
Next, for a bounded obstacle $D\subset \mathbb{R}^2$ we denote by $\overline{D}$ the closure of $D$ and by $D^c := \mathbb{R}^2 \setminus \overline{D}$ the exterior of $D$. In this article, we  assume that $D^c$  is connected. We also denote by $\nu(x)$ the outward unit normal to $\partial D$ at $x\in \partial D$ and $\partial_\nu u(x):=(\nu\cdot \nabla u_1(x),\ \nu\cdot \nabla u_2(x))$, which stands for the normal derivative of the vector field $u$ at $x\in\partial D$. Building on the above notation and definitions, we have that the propagation of time-harmonic elastic wave equations in an isotropic homogeneous elastic medium outside $D$ with the traction-free boundary condition is modeled by 
\begin{align}\label{Traction free Navier equation}
\begin{aligned}
\begin{cases}
\mathcal{N}_{\omega}u(x)=0,\ \ u(x):=u^{in}(x)+u^{sc}(x),\ \ x\in\mathbb{R}^{2}\setminus{\overline{D}},\\
T_{\nu}(u):= 2\mu\, \partial_\nu u 
+ \lambda\,(\nabla \cdot u)\nu
+ \mu\,(\nabla \times u)\, \nu^{\perp} = 0,\ \text{on } \partial D, 
\end{cases} 
\end{aligned}
\end{align} 
where $u^{in}$ represents  the incident field that satisfies \eqref{Equation for Incident field}, $T_{\nu}$ denotes the surface traction operator on $\partial D$, $u$ is known as the total field, and $u^{sc}$ denotes the scattered field.  
Next, we observe that   any solution $u$ of the first Equation in  {\eqref{Traction free Navier equation}} admits the following Helmholtz decomposition
$$
u = u_p + u_s,
$$
where
\begin{equation*}
u_p := -\frac{1}{k_p^2} \nabla(\nabla \cdot u^{in})-\frac{1}{k_p^2} \nabla(\nabla \cdot u^{sc})
\ \ \mbox{and}\ \  
u_s := \frac{1}{k_s^2} \nabla \times (\nabla \times u^{in})+\frac{1}{k_s^2} \nabla \times (\nabla \times u^{sc}),
\end{equation*}
denote the compressional (P-wave) and shear (S-wave) part of $u$, respectively. Also $k_p := \omega/\sqrt{\lambda + 2\mu}$ and $k_s := \omega/\sqrt{\mu}$ are called the compressional and shear wave numbers, respectively.
Moreover, it can be shown that $u^{in}_{p}:=-k_p^{-2}\,\nabla(\nabla \cdot u^{in})$, $u^{in}_{s}:=k_s^{-2}\,\nabla \times (\nabla \times u^{in})$, $u^{sc}_{p}:=-k_p^{-2}\,\nabla(\nabla \cdot u^{sc})$ and $u^{sc}_{s}:=k_s^{-2}\,\nabla \times (\nabla \times u^{sc})$
satisfy the following set of equations 
\begin{equation*}
\begin{aligned}
\label{eq:Helm}
(\Delta + k_p^2)u^{sc}_p=0,\ (\Delta + k_s^2)u^{sc}_s=0
\ \  \mbox{and} \ \ 
\nabla \times u^{sc}_p=\nabla \cdot u^{sc}_s=0\ {\mbox{in}}\ \  \  D^c,\\
(\Delta + k_p^2)u^{in}_p=0,\ (\Delta + k_s^2)u^{in}_s=0
\ \  \mbox{and} \ \ 
\nabla \times u^{in}_p=\nabla \cdot u^{in}_s=0\ \mbox{in}\  \ \  \mathbb{R}^2,
\end{aligned} 
\end{equation*}
whenever the total displacement field $u$ satisfies the first Equation in \eqref{Traction free Navier equation}.

 \noindent In view of the fact that $ D^c$ is unbounded, a radiation condition at infinity must be prescribed for the scattered field $u^{sc}$ in order to ensure the well-posedness of the elastic scattering problem.  The precise statement of the Kupradze radiation condition is given below.
\begin{definition}
    The scattered field $u^{sc}$ corresponding to the first equation of \eqref{Traction free Navier equation} is said to be a Kupradze radiating solution if its compressional part $u_p^{sc}$ and shear part $u_s^{sc}$ fulfill the Sommerfeld radiation condition for the vector Helmholtz equation, that is, 

\begin{equation}
\label{eq:15}
\lim_{r \to \infty} \sqrt{r}\left( \frac{\partial u_{p}^{\mathrm{sc}}}{\partial r} - i k_{p} u_{p}^{\mathrm{sc}} \right) = 0, \  \  
\lim_{r \to \infty} \sqrt{r}\left( \frac{\partial u_{s}^{\mathrm{sc}}}{\partial r} - i k_{s} u_{s}^{\mathrm{sc}} \right) = 0, \  \   r = |x|,
\end{equation}
and the above limits hold uniformly as $\widehat{x}=x/|x|$ ranges over $\mathbb{S}^{1} := \{ x \in \mathbb{R}^2 : |x|=1 \}$.
\end{definition}
 \noindent The elastic scattering of the incident field $u^{in}$ satisfying \eqref{Equation for Incident field}, from an  impenetrable obstacle $D \subset \mathbb{R}^2$ with Lipschitz boundary $\partial D$, is modeled by the following system of equations:
\begin{align}
\label{eq:7}
   & \mathcal{N}_{\omega}u(x)=0,\ \ u(x):=u^{in}(x)+u^{sc}(x),\ \ x\in\mathbb{R}^{2}\setminus{\overline{D}},\\
   \label{eq:8}
& T_{\nu}(u):= 2\mu\, \partial_\nu u 
+ \lambda\,(\nabla \cdot u)\nu
+ \mu\,(\nabla \times u)\, \nu^{\perp} = 0\ \text{on } \partial D, \\
\label{eq:9}
&\lim_{r \to \infty} \sqrt{r}\left( \frac{\partial u_{p}^{\mathrm{sc}}}{\partial r} - i k_{p} u_{p}^{\mathrm{sc}} \right) = 0,\ \mbox{and} \ \lim_{r \to \infty} \sqrt{r}\left( \frac{\partial u_{s}^{\mathrm{sc}}}{\partial r} - i k_{s} u_{s}^{\mathrm{sc}} \right) = 0, \ r = |x|, 
\end{align} 
where, $ \mathcal{N}_{\omega}$ denotes the Navier operator defined in \eqref{Navier operator}, the boundary condition \eqref{eq:8} corresponds to the traction-free condition on $\partial D$ and the Equation   \eqref{eq:9} represents the Kupradze radiation conditions for the compressional part $u^{sc}_{p}$ and shear part $u^{sc}_{s}$, of  the scattered field $u^{sc}$. Also the limits in Equation \eqref{eq:9}  hold uniformly in all directions $\widehat{x}=x/|x|\in\mathbb{S}^1$. 

\noindent Now it is well known (see, for example, \cite{ammari2015mathematical,bao2018direct,kiriaki1999uniquely}) that the above-mentioned scattering problem \eqref{eq:7}-\eqref{eq:9} has a unique solution   $u\in H_{loc}^1(D^c)$ and due to the radiation conditions given by Equation \eqref{eq:9},  
the  scattered field $u^{sc}$ admits the following asymptotic behavior (see \cite{alves2002far}, \cite{hahner1993uniqueness}, \cite{kupradze2012three})
\begin{equation}
\label{eq:asymptotic}
\begin{aligned}
u^{sc}(x) = 
\frac{e^{i k_p |x|}}{\sqrt{|x|}}\, u_p^\infty(\widehat{x})\, \widehat{x}
\;+\;
\frac{e^{i k_s |x|}}{\sqrt{|x|}}\, u_s^\infty(\widehat{x})\, \widehat{x}^\perp
\;+\; 
O(|x|^{-3/2}), \  \   |x| \to \infty,     
\end{aligned}
\end{equation}
which holds uniformly in all directions $\widehat{x} \in \mathbb{S}^1$. The functions $u_p^\infty$ and $u_s^\infty$ defined on $\mathbb{S}^1$, are known as the far-field patterns of $u_p^{sc}$ and $u_s^{sc}$, respectively. The  far-field pattern $u^{\infty}$ of the scattered field $u^{sc}$ is given by 
$u^\infty(\widehat{x}) := u_p^\infty(\widehat{x})\widehat x + u_s^\infty(\widehat{x})\widehat{x} ^{\perp}$. In view of this, we find that far-field patterns  $u_p^{\infty}$ and 
$u_s^{\infty}$ of  $u_p^{sc}$ and $u_s^{sc}$, respectively, can be determined in terms of  far-field pattern $u^{\infty}$ of $u^{sc}$ by  
\begin{equation*}
    \ u_p^{\infty}(\widehat x) = u^{\infty }(\widehat x)\cdot \widehat x , \ \mbox{and}\   u_s^{\infty}(\widehat x) = u^{\infty}(\widehat x) \cdot \widehat x^{\perp}\ \mbox{for}\  \widehat{x}\in\mathbb{S}^1.
\end{equation*}

\noindent For a given obstacle $D \subset\mathbb{R}^2$ and an incident field $u^{in}$ satisfying \eqref{Equation for Incident field}, the associated direct scattering problem in elasticity is to determine the scattered field $u^{sc}$ and their far-field patterns. The inverse scattering problem, which is the main focus of this work, is to determine the shape of the obstacle $D$ from knowledge of the far-field pattern $u^\infty$ associated with a given incident field $u^{in}$.

\noindent In this article, we consider the incident field  $u^{in}$ to be a time-harmonic elastic plane wave of the form
 \begin{equation}
     \label{eq:plane}
     u^{in}(x) = c_pu_p^{in}(x) + c_su_s^{in}(x)
  \end{equation}
  where
$$
u^{in}_p(x) := d\, e^{i k_p x \cdot d}, \  
u^{in}_s(x) := d^\perp e^{i k_s x \cdot d}
$$
represent the compressional and shear components, respectively. Here, the incident direction is given by  
$d = (\cos\theta,\sin\theta)^{T}$ for some $\theta\in[0,2\pi)$, and 
$d^\perp := (-\sin\theta,\cos\theta)^{T}$ denotes the unit vector orthogonal to $d$. The coefficients $c_p, c_s \in \mathbb{C}$ satisfy $|c_p| + |c_s| \neq 0$, ensuring that the incident field is non-trivial.

 \noindent We establish a uniqueness result for the inverse elastic scattering problem described above. More precisely, our objective is to determine a polygonal obstacle $D \subset \mathbb{R}^2$ in the scattering problem \eqref{eq:7}-\eqref{eq:9} by a single far-field pattern $u^{\infty}(\hat{x})\ \mbox{for}\ \hat{x} \in \mathbb{S}^1$, generated by the incident plane wave \eqref{eq:plane} with a fixed incident direction $d \in \mathbb{S}^1$ and a fixed frequency $\omega >0$. We assume that the obstacle $D$ is a bounded polygonal domain such that the exterior domain $D^c$ is connected. Moreover, the boundary $\partial D$ consists of finitely many straight line segments, so $D$ represents a single polygonal region without cracks. 
 
 The main result of this article is the following uniqueness theorem.
\begin{theorem}
\label{thm:main_result}
Let $D_1 , D_2 \subset \mathbb{R}^2 $ be two bounded polygonal obstacles such that $D_1^c$ and $D_2^c$ are connected. For $j=1, 2$, let $u_j^\infty$ denote the elastic far-field pattern corresponding to the scattering problem \eqref{eq:7}-\eqref{eq:9}, generated by an incident plane wave \eqref{eq:plane} with a fixed incident direction ${d} \in \mathbb{S}^1$ and a fixed frequency $\omega > 0$. If 
\begin{align*}
u_1^\infty(\widehat{x}) = u_2^\infty(\widehat{x}) \  \text{for all } \widehat{x} \in \mathbb{S}^1,
\end{align*}
then 
\begin{align*}    
D_1 = D_2.
\end{align*}
\end{theorem}
\subsection{Literature reviews} 
Inverse scattering problems have attracted extensive attention because of their importance in applications such as geophysical exploration, nondestructive testing, and medical imaging \cite{colton2019inverse}. A central theoretical question in this field is the uniqueness problem, namely whether the shape and the physical properties of an unknown obstacle are uniquely determined by the measured scattered data. For a single incident wave, this question is closely connected with the analytic properties of the scattered field and is still open in general, both for the inverse acoustic and for the inverse elastic scattering problem: uniqueness with a single incident wave is known only under additional assumptions on the geometry of the obstacle or on the type of the boundary condition (see, e.g., \cite{colton1977reflection}, \cite{lax1990scattering}, \cite{hahner1993uniqueness}, \cite{elschner2010uniqueness}, \cite{colton2019inverse}, \cite{elschner2019uniqueness}).

For the inverse acoustic and electromagnetic scattering problems, the unique continuation argument, in particular the reflection (mirror) argument, has proved to be a powerful and successful tool for establishing uniqueness with a single incident wave. It yields uniqueness for obstacles possessing a reflection symmetry, such as disks \cite{colton1977reflection}, \cite{lax1990scattering}, and, more importantly, for polygonal and polyhedral obstacles, for which the reflection principle for the Helmholtz equation is applied across the straight edges of the obstacle and then combined with the path argument \cite{cheng2003uniqueness}, \cite{alessandrini2005determining}, \cite{alessandrini2006corrigendum}, \cite{elschner2006uniqueness}, \cite{liu2006uniqueness}, \cite{elschner2008uniqueness}. Closely related is the corner-scattering phenomenon, which shows that the corners and edges of a polygonal or polyhedral obstacle always scatter and can therefore be detected from the far-field pattern \cite{blasten2014corners}, \cite{elschner2015corners}.

The elastic case is considerably more subtle, mainly because the pressure (P) waves and the shear (S) waves governed by the Navier system are coupled. Extensive results on inverse elastic scattering from impenetrable obstacles have been obtained in recent decades. Uniqueness using the far-field pattern $u^{\infty}$ corresponding to infinitely many incident plane waves was established in \cite{hahner1993uniqueness}, \cite{hu2013some} with full-phase data and in \cite{cheng2024uniqueness} with phaseless data. For general elastic obstacles, however, uniqueness using only a single incident wave remains a challenging open problem, and only partial results are available under additional geometric assumptions. For instance, it was shown in \cite{hu2013some} that a rigid ball and a convex Lipschitz polyhedron are uniquely determined by the shear (S) part of the far-field pattern corresponding to a single incident shear wave, while the first uniqueness result for a non-convex polyhedral obstacle was obtained in \cite{elschner2010uniqueness} with at most two incident plane waves for the third and fourth exterior boundary value problems of the Navier equation, in which the normal stress and the tangential displacement vanish on $\partial D$. That approach is motivated by the earlier developments for acoustic and electromagnetic polygonal obstacles \cite{diaz1955reflection}, \cite{cheng2003uniqueness}, \cite{alessandrini2006corrigendum}, \cite{elschner2006uniqueness}, \cite{liu2006uniqueness}, \cite{elschner2008uniqueness} and combines a reflection principle for the Navier equation with the path argument; however, it depends crucially on the particular form of the third and fourth boundary conditions, and its extension to the first and second boundary conditions is not applicable, owing to the lack of an appropriate reflection principle for the Navier equation in that setting. Subsequent progress for the first boundary condition was made in \cite{elschner2019uniqueness}, where different analytical techniques were introduced to overcome this difficulty. More recently, the same circle of ideas has been extended to obstacles coated by a thin layer and modeled by an impedance-type boundary condition: with a single incoming wave, uniqueness was obtained for coated polygonal acoustic obstacles in \cite{hu2020uniqueness} and for coated polyhedral electromagnetic scatterers in \cite{hu2021inverse}; related uniqueness results for electromagnetic problems can be found in \cite{niu2023uniqueness}.

Unlike the acoustic and electromagnetic cases, the reflection principle for the Navier equation under the traction-free boundary condition has only recently become available, and it is precisely this principle that underlies the present article. It is also essentially different from the classical reflection principle for the Helmholtz equation under the Dirichlet or Neumann boundary condition. The latter is of point-to-point type: the reflected value of the field at a point is expressed solely in terms of the value of the field at its mirror image with respect to the reflecting line. The reflection for the Navier (Lam\'e) system, by contrast, is \emph{not} of point-to-point type: the extended field at a point is given by an integral-type formula involving the values of the field and of its derivatives along the whole boundary segment $\gamma$, rather than at a single mirror point \cite{duffin1956analytic}, \cite{elschner2019uniqueness}. It is precisely this non-local character of the elastic reflection that makes the proof of uniqueness for the Navier equation substantially more difficult than in the acoustic case, and it is the main technical obstacle that the present article overcomes.

As a by-product of our uniqueness proof, we further show that the total field $u$ and the scattered field $u^{sc}$ of the traction-free elastic scattering problem cannot be real-analytic around any corner of a polygonal obstacle; see Corollary~\ref{Co-3.5}. This non-analyticity at the corners is the counterpart, for the Navier equation, of the corner-scattering phenomenon for the Helmholtz equation \cite{elschner2015corners}, and it provides a theoretical basis for the numerical inversion of the corners of the obstacle: by locating the points at which the analytic continuation of the wave field breaks down, one can recover the positions of the corner points and thereby reconstruct the shape of a polygonal obstacle. Reconstruction methods relying on the analytic continuation (and on reflection principles) of the wave field have been developed for the inverse acoustic and elastic scattering problems; in particular, the single-wave factorization method for inverse elastic scattering proposed in \cite{elschner2019uniqueness} exploits precisely this kind of analytic-continuation argument, and the same circle of ideas underlies other continuation-based reconstruction methods. We expect that the reflection principle and the corner-based continuation established in this article can be used in a similar fashion to design effective numerical schemes for reconstructing polygonal traction-free elastic obstacles.

The remainder of the paper is organized as follows. In Section~\ref{sec:reflection}, we establish reflection principles for the Navier equation under traction-free boundary conditions, first for the Lam\'e system and then for the Navier operator. In Section~\ref{sec:application}, we apply these reflection principles to analytically extend solutions of the traction-free boundary value problem across the edges and corners of a polygonal obstacle, and we use the resulting continuation argument to prove the uniqueness theorem (Theorem~\ref{thm:main_result}) stated in Section~\ref{mathematical formulation}.

\section{Reflection principle}\label{sec:reflection}
 We begin this Section by fixing some notations which will be  required  while deriving the reflection  principle.
 Let $R_\Lambda$ denote the reflection operator with respect to the  line $\Lambda:=\{(x_1,x_2) \in \mathbb{R}^2  :  x_1 = 0 \}\subset\mathbb{R}^2$, which is given by 
\[R_\Lambda(x_1, x_2):=(-x_1, x_2)\ \mbox{for} \ (x_1,x_2) \in \mathbb{R}^2.\]
We assume that $\Omega \subset \mathbb{R}^2$ is an open connected and  symmetric with respect to the line $\Lambda$ in the sense that 
 $R_\Lambda(x_1, x_2) \in \Omega$  whenever   $(x_1, x_2) \in \Omega$. In addition, throughout this Section we assume that every line segment with endpoints in $\Omega$ which intersects $\Lambda$ orthogonally is entirely contained in $\Omega$. Due to the symmetry assumption on $\Omega$ with respect to  $\Lambda$, we observe that $\Omega$ can be decomposed as 
 \[
\Omega = \Omega^+ \cup \gamma \cup \Omega^- \] where \[
\Omega^{\pm}  := \Omega \cap \{(x_1,x_2)\in\mathbb{R}^2 : \pm x_1>0\} \ \mbox{and}\  \gamma:=\Omega\cap \Lambda=\{(x_1,x_2)\in\Omega:\ x_1=0\}.\]

The reflection geometry of $\Omega$, $\Omega^{\pm}$ and $\gamma$ is illustrated in Figure~\ref{fig:omega_symmetric}.

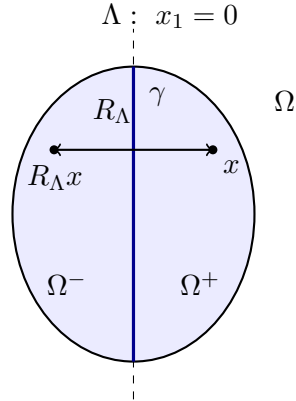
\begin{figure}[H]
\centering
\begin{tikzpicture}[scale=1.0]
  \fill[blue!8] (0,0) ellipse (1.6 and 1.95);
  \draw[thick] (0,0) ellipse (1.6 and 1.95);
  \draw[dashed] (0,-2.45) -- (0,2.45);
  \draw[very thick,blue!55!black] (0,-1.95) -- (0,1.95);
  \coordinate (P) at (1.05,0.85);
  \coordinate (Q) at (-1.05,0.85);
  \fill (P) circle (1.7pt);
  \fill (Q) circle (1.7pt);
  \draw[<->,thick] (P) -- (Q);
  \node[above,  xshift=-8pt,inner sep=5pt] at (0,1.02) {$R_\Lambda$};
  \node[below right] at (P) {$x$};
  \node[below right,xshift=-14pt,yshift=-2pt] at (Q) {$R_\Lambda x$};
  \node at (0.88,-0.92) {$\Omega^{+}$};
  \node at (-0.88,-0.92) {$\Omega^{-}$};
  \node[right] at (0.07,1.55) {$\gamma$};
  \node at (0.48,2.62) {$\Lambda:\ x_1=0$};
  \node at (2.0,1.5) {$\Omega$};
\end{tikzpicture}
\caption{The symmetric domain $\Omega=\Omega^{+}\cup\gamma\cup\Omega^{-}$: the line $\Lambda=\{x_1=0\}$ is the axis of symmetry, $\gamma=\Omega\cap\Lambda$ is the common boundary, and the two halves $\Omega^{+}=\Omega\cap\{x_1>0\}$ and $\Omega^{-}=\Omega\cap\{x_1<0\}$ are mapped onto each other by the reflection $R_\Lambda$.}
\label{fig:omega_symmetric}
\end{figure}

The present section is divided into  two subsections. The first subsection  deals with the  establishing of reflection principle for Lam\'e system  and in the second subsection, we use this result to prove the reflection principle for the Navier equation. Both reflection principles are derived under traction-free boundary conditions. Before moving forward, we begin with observing that whenever $u$ solves \[\mathcal{L}u:=\mu \Delta u +(\lambda+\mu)\nabla(\nabla\!\cdot u)=0\ \mbox{in}\ D,\] we obtain that 
\[
\mu \Delta u = -(\lambda+\mu)\nabla(\nabla\!\cdot u)\ \mbox{in}\ D,\] where  $D\subset\mathbb{R}^2$ is an open set and $u=(u_1,u_2)$, a vector field defined on $D$.  
Hence, under the assumptions on $\lambda,\mu$ as mentioned in Section \ref{mathematical formulation}, we arrive at 
\[
\Delta u = -\frac{\lambda+\mu}{\mu}\nabla(\nabla\!\cdot u)\ \mbox{in}\ D.\]
Now if we introduce a  scalar function 
\begin{equation}\label{ex-f}
f(x) := -\frac{\lambda+\mu}{\mu}(\nabla\!\cdot u)(x)\ \text{in} \  D.
\end{equation}
Using these notations and aforementioned assumptions on $\lambda,\mu$, we immediately get that 
\begin{align}\label{System of equations}
\Delta u = \nabla f, \ \nabla\!\cdot u = -cf, \ \mbox{and}\ \Delta f=0\ \mbox{in}\ D,
\end{align}
whenever $\mathcal{L}u=0$ in $D$. 
Thus solving $\mathcal{L}u=0$ in $D$ is equivalent to solving the system \eqref{System of equations}. 
Motivated by this, 
we first state the following lemma, which will be required to derive the reflection principle for the Lam\'e system. For a detailed proof of the following lemma, we refer to  \cite{duffin1956analytic}. 
 \begin{lemma}[{{\cite[Theorem-2]{duffin1956analytic}}}]\label{lemma:1}
Let $\Lambda$, $\gamma$ and the domains $\Omega^\pm$ be defined as above. Suppose that $u=(u_1,u_2):\Omega^+\to\mathbb{R}^2$ and $f:\Omega^+\to\mathbb{R}$
satisfy
\begin{equation}
\label{eq:duffin-system}
\begin{cases}
\Delta u(x) = \nabla f(x)\ \ \mbox{and}\  \ 
\nabla\cdot u(x) = -c f(x)\ \text{for}\  x\in\Omega^+,\\
u(x)=0\ \mbox{for}\  x\in\gamma, 
\end{cases}
\end{equation}
where  $\Delta f=0$ in $\Omega^+$. 
Then the functions $u$ and $f$ admit a continuation to $\Omega$ satisfying
\eqref{eq:duffin-system} throughout $\Omega$. This continuation satisfies the identities 
\begin{align*}
 u_1^{\prime}(x_1,x_2)
&=
- u_1(x_1,x_2)
+ \frac{x_1}{1+2c}
\bigl(
2c\, f(x_1,x_2)
+ 2\,\partial_1 u_1(x_1,x_2)
- x_1\,\partial_1 f(x_1,x_2)
\bigr), \\
 u_2^{\prime}(x_1,x_2)
&=
- u_2(x_1,x_2)
+ \frac{x_1}{1+2c}
\bigl(
-2\,\partial_2 u_1(x_1,x_2)
+ x_1\,\partial_2 f(x_1,x_2)
\bigr), \\
f^{\prime}(x_1,x_2)
&=
\frac{1}{1+2c}
\bigl(
f(x_1,x_2)
- 2c\, f(x_1,x_2)
- 4\,\partial_1 u_1(x_1,x_2)
+ 2x_1\,\partial_1 f(x_1,x_2)
\bigr).
\end{align*}
Here $u^{\prime}$ denotes $u(-x_1,x_2)$, etc. It is assumed that $c\neq -1,-\tfrac12$.
\end{lemma}
\subsection{Reflection principle for Lam\'e system} The present Subsection is devoted to establishing the reflection principle for the Lam\'e system stated in the following theorem. 
\begin{theorem}
\label{thm:reflection}
Suppose that $u(x_1,x_2)=(u_1(x_1,x_2),u_2(x_1,x_2))$ is a solution to  the following Lam\'e system with the traction-free boundary  condition 
\begin{equation}
\label{Reflection: Lame system traction free condition}
\begin{cases}
\mathcal{L}u=\mu \Delta u + (\lambda + \mu)\nabla(\nabla \cdot u)  = 0\quad \text{in }\; \Omega^+, \\
 T_{\nu}(u)=2\mu\, \partial_\nu u + \lambda\, (\nabla \cdot u)\nu 
+ \mu\, (\nabla \times u)\, \nu^{\perp} = 0\quad \text{on }\; \gamma.
\end{cases}
\end{equation}
Then $u$ can be analytically extended from $\Omega^+$ to  $\Omega$ as a solution to \eqref{Reflection: Lame system traction free condition}. Moreover, the extended field $\widetilde u(x_1,x_2)=(\widetilde u_1(x_1,x_2),\widetilde u_2(x_1,x_2))$ on $\Omega$  is explicitly given by 
 \begin{equation}
 \label{eqn:Extension}
    \widetilde u(x_1,x_2) :=
        \begin{cases}
          u(x_1,x_2) &\text{  if  } (x_1, x_2) \in \Omega^+ \cup \gamma,\\
          \mathcal{B}u(R_\Lambda (x_1,x_2)) &\text{ if  } (x_1,x_2) \in \Omega^-,
        \end{cases}
    \end{equation}
    where $\mathcal{B}$ is the linear operator mapping the space of analytic functions in $\Omega^+$ into itself and given by $\mathcal{B}u=((\mathcal{B}u)_1,(\mathcal{B}u)_2)$ where, for $(x_1,x_2)\in \Omega^+$,
    \begin{align}\label{eq:2.5}
 (\mathcal{B}u)_1(x_1,x_2)
&=
\frac{\lambda+3\mu}{\lambda+\mu}\,
u_1(x_1,x_2)
-\frac{2\mu}{\lambda+\mu}\,
u_1(0,x_2)
+\frac{\mu(\lambda+2\mu)}{(\lambda+\mu)^2}
\int_0^{x_1} f(s,x_2)\,ds \notag
\\
&\quad
-2x_1\,\partial_1u_1(x_1,x_2)
-\frac{2\mu}{\lambda+\mu}\,
x_1f(x_1,x_2)
+x_1^2\,\partial_1f(x_1,x_2),
\end{align}
and 
\begin{align}\label{eq:2.6}
(\mathcal{B}u)_2(x_1,x_2)
&=
u_2(x_1,x_2)
+\frac{2\mu}{\lambda+\mu}
\int_0^{x_1}\partial_2u_1(s,x_2)\,ds
-\frac{2\mu}{\lambda+\mu}\,
x_1\partial_2u_1(0,x_2) \notag \\ & \ \
+2x_1\partial_2u_1(x_1,x_2)
+\frac{\mu(\lambda+2\mu)}{(\lambda+\mu)^2}
\int_0^{x_1}(x_1-t)\,
\partial_2f(t,x_2)\,dt
\notag \\
&\quad
-x_1^2\partial_2f(x_1,x_2).
\end{align}
Here the function $f$ is given by \eqref{ex-f}.
\end{theorem} 
\begin{proof}
The traction-free condition does not force $u$ to vanish on the flat part $\gamma$, so Duffin\'s analytic continuation lemma cannot be applied to $u$ directly. The key point is that the traction $v:=T_\nu u$ vanishes on $\gamma$ and that the pair $(v,g)$, formed from $v$ and the harmonic function $f$, satisfies the hypotheses of Duffin\'s lemma. Applying it gives an explicit continuation of $(v,g)$ across $\gamma$. Since $(u,f)$ and $(v,g)$ are linked by the first-order system \eqref{System of equations}, integrating that system reconstructs the continuation of $(u,f)$, which is then verified to solve the Lam\'e system in all of $\Omega$. For clarity, we divide the proof into six steps.

\medskip
\noindent \textbf{Step 1 (Equivalent first-order formulation).}
 Let $u=(u_1,u_2)$ be a solution of the Lamé system
 \begin{equation}\label{eq:lame_recall}
\mu \Delta u + (\lambda+\mu)\nabla(\nabla\!\cdot u)=0\ \text{in } \Omega^+,
\end{equation}
satisfying the traction--free boundary condition 
\begin{equation}\label{eq:traction_recall}
T_{\nu}(u)=2\mu\,\partial_\nu u+\lambda(\nabla\!\cdot u)\nu
+\mu(\nabla\times u)\nu^\perp=0
\  \text{on } \gamma.
\end{equation}
As explained earlier, we have that the  equation \eqref{eq:lame_recall} can be written in the equivalent form
\begin{equation}\label{eq:equivalent_form}
\Delta u = \nabla f, \ 
\nabla\!\cdot u = -cf, \ \mbox{and}\  \Delta f=0\ \mbox{in}\ \Omega^{+},\ \mbox{where $c=\frac{\mu}{\lambda+\mu}$}. 
\end{equation}
 \medskip
\noindent \textbf{Step 2 (The auxiliary field $v$ and reduction to Duffin\'s system).} Now since $u$ satisfies the traction–free boundary condition \eqref{eq:traction_recall} on $\gamma$, the displacement field $u$ does not necessarily vanish on $\gamma$. 
Consequently, unlike in \cite{elschner2019uniqueness}, Lemma~\ref{lemma:1} cannot be applied directly to $u$. To overcome this difficulty, we introduce a new vector field constructed from the traction operator on $\Omega^{+}$ by 
\begin{equation} \label{eq T operator}
   v:= 2\mu\,\partial_\nu u+\lambda(\nabla\!\cdot u)\nu
+\mu(\nabla\times u)\nu^\perp, 
\end{equation}
 with the understanding that the outward unit normal $\nu$ to $\gamma$ is extended as a smooth function on $\Omega^{+}$ and $u$ solves \eqref{eq:lame_recall}-\eqref{eq:traction_recall}.
Without loss of generality, we may assume that
$$
\gamma \subset \Lambda := \{(x_1,x_2)\in\mathbb{R}^2 : x_1=0\},
$$
with outward unit normal $\nu=(1,0)$ and tangential vector $\nu^\perp=(0,1)$.
Indeed, a direct computation then shows that
$$
\partial_\nu u = \begin{bmatrix}
   \partial_1 u_1 \\ \partial_1 u_2 
\end{bmatrix},
\ 
(\nabla\!\cdot u)\nu= \begin{bmatrix}
    \partial_1 u_1+\partial_2 u_2 \\ 0
\end{bmatrix},
$$
and
$$
(\nabla\times u)\nu^\perp= \begin{bmatrix}
    
0 \\ \partial_2 u_1-\partial_1 u_2
\end{bmatrix}.
$$
Substituting these expressions in \eqref{eq T operator} and using the fact that $u$ solves \eqref{eq:lame_recall}-\eqref{eq:traction_recall},  we get that 
\begin{align}
\begin{aligned}
\begin{cases}
    v= 2\mu \begin{bmatrix}
\partial_1 u_1 \\ \partial_1 u_2         
    \end{bmatrix} + \lambda \begin{bmatrix}
        \partial_1 u_1 + \partial_2 u_2 \\ 0
    \end{bmatrix} + \mu \begin{bmatrix}
        0 \\ \partial_2 u_1 - \partial_1 u_2
    \end{bmatrix}
    &= \begin{pmatrix}
(\lambda+2\mu)\partial_1 u_1+\lambda\partial_2 u_2\\[4pt]
\mu(\partial_2 u_1+\partial_1 u_2)
\end{pmatrix}\ \mbox{in}\ \Omega^{+},\\
v=0\ \mbox{on}\ \gamma.
\end{cases} 
\end{aligned} 
\end{align}
Now if we define
$$
v(x):= \begin{bmatrix}
    v_1(x) \\ v_2(x)
\end{bmatrix}, \ \ x\in \Omega^+
$$ then 
\begin{equation}\label{eq: new vector field}
\begin{cases}
v_1(x_1, x_2):=(\lambda+2\mu)\partial_1u_1(x_1,x_2)+\lambda\partial_2 u_2(x_1,x_2)\quad \mbox{for}\ (x_1,x_2)\in\Omega^{+}\\ 
v_2(x_1,x_2):=\mu(\partial_2 u_1(x_1, x_2)+\partial_1 u_2(x_1, x_2))\quad \mbox{for}\  (x_1,x_2)\in\Omega^{+}\\
v_1(x_1,x_2)=v_2(x_1,x_2)=0\quad \mbox{for}\ (x_1,x_2)\in \gamma. 
\end{cases}
\end{equation}
 We next verify that the vector field $v=(v_1,v_2)$ defined above, satisfies the assumptions of Lemma~\ref{lemma:1}, that is, 
\begin{equation}
\begin{cases}
\Delta v(x) = \nabla g(x),\ 
\nabla\cdot v(x) = -c_1 g(x), \ x\in \Omega^{+}\\
v(x)=0,\ x\in \gamma
\end{cases}
\end{equation} for some scalar function $g$  satisfying   $\Delta g=0,$ in $\Omega^+$, and the constant $c_1\neq -1,-\tfrac12$. Below we will provide the  detailed expressions for the aforementioned function $g$ and constant $c_1$. 
 Now definition of $v$, yields that  
$$
\nabla\!\cdot v=\partial_1 v_1+\partial_2 v_2.
$$
Using \eqref{eq: new vector field}, we compute
\begin{align*}
\nabla\!\cdot v
&=\partial_1\!\left((\lambda+2\mu)\partial_1 u_1+\lambda\partial_2 u_2\right)
 +\partial_2\!\left(\mu(\partial_2 u_1+\partial_1 u_2)\right)\\
&=(\lambda+2\mu)\partial_1^2 u_1
 +\lambda\partial_{12} u_2
 +\mu\partial_2^2 u_1
 +\mu\partial_{12} u_2\\
 &= \lambda \partial_1^2 u_1 + \mu \partial_1^2 u_1 + \mu \partial_1^2 u_1 + \lambda \partial_{12} u_2 + \mu \partial_2^2 u_1 + \mu \partial_{12} u_2\\
 &= \lambda \partial_1^2 u_1 + \mu \Delta u_1 + \mu \partial_1^2 u_1 + \lambda \partial_{12} u_2 + \mu \partial_{12} u_2\\
 &=(\lambda + \mu)\partial_1^2 u_1 + \mu \Delta u_1 + (\lambda + \mu) \partial_{12} u_2\\
 &=(\lambda + \mu)\partial_1(\partial_1 u_1 + \partial_2 u_2) + \mu \Delta u_1=(\lambda + \mu)\partial_1(\nabla \cdot u) + \mu \Delta u_1. 
\end{align*}
Thus, we get 
\begin{equation} \label{eq div v}
    \nabla\!\cdot v
=\mu\Delta u_1+(\lambda+\mu)\partial_1(\nabla\!\cdot u).
\end{equation}
Invoking \eqref{eq:equivalent_form}, namely
$
\Delta u_1=\partial_1 f
$ and $\nabla\!\cdot u=-cf $ with $c=\mu/(\lambda+\mu)$ in \eqref{eq div v},  we obtain
\begin{equation}\label{divfree}
 \nabla\!\cdot v
=\mu\partial_1 f-(\lambda+\mu)c\,\partial_1 f
=0\quad \text{in } \Omega^+.   
\end{equation}
Next, we compute the Laplacian of $v$ component-wise.
For the first component,
\begin{align*}
\Delta v_1
=(\lambda+2\mu)\partial_1(\Delta u_1)
 +\lambda\partial_2(\Delta u_2)=(\lambda+2\mu)\partial_1^2 f+\lambda\partial_2^2 f
=\lambda\Delta f+2\mu\partial_1^2 f.
\end{align*}
Utilizing the fact that  $f$ is harmonic in $\Omega^+$, we get 
$$
\Delta v_1=2\mu\partial_1^2 f.
$$
Similarly, for the second component,
\begin{align*}
\Delta v_2
=\mu\partial_2(\Delta u_1)+\mu\partial_1(\Delta u_2)=\mu\partial_2(\partial_1 f)+\mu\partial_1(\partial_2 f)
=2\mu\partial_{12} f.
\end{align*}
Consequently,
\begin{align} \label{eq:v_system}
\begin{aligned}
\Delta v =\begin{bmatrix}
    \Delta v_1 \\ \Delta v_2
\end{bmatrix}
=\begin{pmatrix}
2\mu\partial_1^2 f\\
2\mu\partial_{12} f
\end{pmatrix} = 2\mu \partial _1 \begin{bmatrix}
  \partial _1f \\ 
  \partial _2 f  
\end{bmatrix}
= 2\mu \partial_1(\nabla f) =\nabla (2\mu \partial_1 f):=
\nabla g\ \mbox{in}\ \Omega^{+} 
\end{aligned}
\end{align}
 where \begin{equation}\label{eq:g_relation}
g(x):=2\mu\,\partial_1 f(x) \ \ \text{in} \ \ \Omega^+.
\end{equation}
Since $f$ is harmonic in $\Omega^+$, we further have
\begin{equation}
    \Delta g
=2\mu\,\partial_1(\Delta f)
=0\quad \text{in } \Omega^+.
\end{equation}
Therefore, the pair $(v,g)$ satisfies the system
\begin{equation}
\label{eq:required}
\begin{cases}
\Delta v(x) = \nabla g(x),\ 
\nabla\cdot v(x) = -c_1 g(x), \ x\in \Omega^{+}\\
v(x)=0,\ x\in \gamma
\end{cases}
\end{equation}
where $g(x) =2\mu\,\partial_1 f(x)$ is harmonic in $\Omega^+$ and from \eqref{divfree} we can take $c_1=0$, such that $\nabla\cdot v(x) = -c_1 g(x)\ \text{for all}\ x\in \Omega^{+}.$  \medskip

\noindent \textbf{Step 3 (Continuation of $(v,g)$ via Duffin\'s lemma).} Now since $c_1\neq -1,-\tfrac12$, all the assumptions of Lemma~\ref{lemma:1} are fulfilled. Hence, $v$ and $g$ admit an analytic continuation from $\Omega^+$ to $\Omega$ through $\gamma$, that satisfies the equations in \eqref{eq:required} throughout $\Omega$. Here $g(x_1,x_2)$ on $\gamma$ is the limit of g as $x_1\to 0$. Moreover,  if we denote by $\widetilde{v}_{j}$ ($1\leq j\leq 2$) and $\widetilde{g}$,  the analytic continuation of $v_j$ ($1\leq j\leq 2$) and $g$, respectively, from $\Omega^{+}$ to $\Omega$, given by 
\begin{equation}\label{eq:tilda}
\begin{cases}
\widetilde{v}_1(-x_1,x_2)
=-v_1(x_1,x_2)
+x_1\bigl(2\partial_1 v_1(x_1,x_2) - x_1\partial_1 g(x_1,x_2)\bigr),\  \  \  (x_1,x_2)\in\Omega^{+},\\\\
\widetilde v_2(-x_1,x_2)
=-v_2(x_1,x_2)
+x_1\bigl(-2\partial_2 v_1(x_1, x_2) + x_1\partial_2 g(x_1, x_2)\bigr),\ \  (x_1,x_2)\in\Omega^{+},\\\\
\widetilde g(-x_1,x_2)
=g(x_1,x_2) - 4\partial_1 v_1(x_1,x_2)
+2x_1\partial_1 g(x_1,x_2),\  \   \  \  \  (x_1,x_2)\in\Omega^{+}.
\end{cases}
\end{equation}
 After substituting $v_1 , v_2$ from
\eqref{eq: new vector field} and  $g(x) =2\mu\,\partial_1 f(x)$ into the identities \eqref{eq:tilda}, and using the relations
$\Delta u=\nabla f$ and $\nabla\!\cdot u=-cf$, we arrive at 
\begin{equation*}
\begin{cases}
\begin{aligned}
\widetilde v_1(-x_1,x_2)
&= -(\lambda +2\mu )\partial_1 u_1(x_1,x_2) -\lambda \partial_2 u_2(x_1,x_2)  + 2(\lambda+ 2 \mu)x_1 \partial_1^2 u_1(x_1,x_2)\\
& \ \ + 2\lambda x_1 \partial_{12} u_2(x_1,x_2) -2\mu x_1^2\partial_1^2f(x_1,x_2), \  \  \  (x_1,x_2)\in\Omega^{+},
\end{aligned}\\\\
\begin{aligned}
\widetilde v_2(-x_1,x_2)
&=-\mu (\partial_2 u_1(x_1,x_2) +  \partial_1 u_2(x_1,x_2)) -2(\lambda +2 \mu)x_1 \partial_ {12} u_1 (x_1,x_2)\\
& \ \ -2 \lambda x_1 \partial _2^2 u_2(x_1,x_2) + 2\mu x_1^2 \partial_{12}f(x_1,x_2), \  \  \  (x_1,x_2)\in\Omega^{+}, 
\end{aligned}\\\\
\begin{aligned}
\widetilde g(-x_1,x_2)
&=2\mu\,\partial_1 f(x_1,x_2)
- 4(\lambda+2\mu) \partial_1^2 u_1(x_1,x_2)
-4\lambda \partial_{12} u_2 (x_1,x_2)
\\& \ \ +4\mu x_1 \partial_1^{2} f(x_1,x_2), \  \  \  (x_1,x_2)\in\Omega^{+}.
\end{aligned}
\end{cases}
\end{equation*} 
 After rearranging the terms and using divergence formula,  the above expression reduces to  
\begin{equation*}
\begin{cases}
\begin{aligned}
\widetilde v_1(-x_1,x_2)
&= -\lambda (\nabla \cdot u(x_1,x_2)) + 2\lambda x_1 \partial_1(\nabla \cdot u(x_1,x_2))- 2\mu \partial_1 u_1(x_1,x_2) \\&
\ \ + 4\mu x_1 \partial_1^2 u_1(x_1,x_2)-2\mu x_1^2\partial_1^2f(x_1,x_2), \  \  \  (x_1,x_2)\in\Omega^{+}, 
\end{aligned}\\\\

\begin{aligned}
\widetilde v_2(-x_1,x_2)
&=-\mu (\partial_2 u_1(x_1,x_2) +\partial_1 u_2(x_1,x_2)) -2\lambda x_1 \partial_2(\nabla \cdot u(x_1,x_2) ) \\& \ \ -4 \mu x_1 \partial_{12} u_1(x_1,x_2) + 2\mu x_1^2 \partial_{12}f(x_1,x_2), \  \  \  (x_1,x_2)\in\Omega^{+},
\end{aligned}
 \\\\

\begin{aligned}
\widetilde g(-x_1,x_2)
&=2\mu\,\partial_1 f(x_1,x_2) -4\lambda \partial_1 (\nabla \cdot u(x_1,x_2))
-8\mu \partial_1^2 u_1(x_1,x_2) \\&
\ \ +4\mu x_1 \partial_1^{2} f(x_1,x_2), \  \  \  (x_1,x_2)\in\Omega^{+}.
\end{aligned}
\end{cases}
\end{equation*}
Next we use  \eqref{eq:equivalent_form}, to obtain 
\begin{equation}\label{eq 30}
\begin{cases}
\begin{aligned}
\widetilde v_1(-x_1,x_2)
&= \lambda c f(x_1,x_2)
-2\lambda c x_1 \partial_1 f(x_1,x_2)
- 2\mu \partial_1 u_1(x_1,x_2)\\&
\ \ + 4\mu x_1 \partial_1^2 u_1(x_1,x_2)-2\mu x_1^2\partial_1^2f(x_1,x_2), \  \  \  (x_1,x_2)\in\Omega^{+},
\end{aligned}\\\\
\begin{aligned}
\widetilde v_2(-x_1,x_2)
&=-\mu (\partial_2 u_1(x_1,x_2) +\partial_1 u_2(x_1,x_2)) +2\lambda c x_1 \partial_2 f(x_1,x_2) \\& \ \  -4 \mu x_1 \partial_{12} u_1(x_1,x_2) + 2\mu x_1^2 \partial_{12}f(x_1,x_2), \  \  \  (x_1,x_2)\in\Omega^{+},
\end{aligned}\\\\
\begin{aligned}
\widetilde g(-x_1,x_2)
&=2\mu \partial_1 f(x_1,x_2)
+4\lambda c \partial_1 f(x_1,x_2)
\\& \ \ -8\mu \partial_1^{2}u_1(x_1,x_2) 
+4\mu x_1 \partial_1^{2} f(x_1,x_2), \  \  \  (x_1,x_2)\in\Omega^{+}.
\end{aligned}
\end{cases}
\end{equation}
\medskip
\noindent \textbf{Step 4 (Continuation of $(u,f)$).}  The Step 3 gives the analytic continuation of $v_j$ ($1\leq j\leq 2$) and $g$ from $\Omega^{+}$ to $\Omega$ across $\gamma$. We next use the analytically continued fields $\widetilde v$ and $\widetilde g$ to determine corresponding
extensions of $u_1$, $u_2$, and $f$ across $\gamma$. More precisely, we seek functions $\widetilde u(x_1,x_2)=(\widetilde u_1(x_1,x_2),\widetilde u_2(x_1,x_2))$ and $\widetilde f(x_1,x_2)$ defined in $\Omega$ such that $\widetilde u_j(x_1,x_2)=u_j(x_1,x_2)$, 
 $(j=1,2)$ and $\widetilde f(x_1,x_2)=f(x_1,x_2)$ in $\Omega^+$.
 Since the fields $(v,g)$ and $(u,f)$ are related by
\eqref{eq: new vector field} and \eqref{eq:g_relation}, we require their continuations
to satisfy the corresponding identities
\begin{equation}\label{eq:2.14}
\begin{cases}
\widetilde v_1(-x_1,x_2)
= 2\mu\,\partial_1\widetilde u_1(-x_1,x_2)-\lambda c\,\widetilde f(-x_1,x_2),  \ \  (x_1,x_2)\in\Omega^{+} \\\\
\widetilde v_2(-x_1,x_2)
=\mu\partial_2\widetilde u_1(-x_1,x_2)+\mu\partial_1\widetilde u_2(-x_1,x_2), \ \  (x_1,x_2)\in\Omega^{+} \\\\
\widetilde g(-x_1,x_2)
=2\mu\,\partial_1\widetilde f(-x_1,x_2), \ \  (x_1,x_2)\in\Omega^{+}.
\end{cases}
\end{equation}
Since $(\widetilde v_1,\widetilde v_2)$ and $\widetilde g$ are already known from \eqref{eq:tilda}, these relations determine a system of first-order differential equations for the unknown functions $(\widetilde u_1,\widetilde u_2) $ and $\widetilde f$.

Comparing the explicit expressions \eqref{eq:2.14} with \eqref{eq 30}, we obtain the following
identities 
\begin{align}
\partial_1\widetilde u_1(-x_1,x_2)
&=\frac{\lambda c}{2\mu} \widetilde f(-x_1,x_2) + \frac{\lambda c}{2\mu} f(x_1,x_2)-\frac{\lambda c}{\mu} x_1 \partial_1 f(x_1,x_2) \notag \\& \ \ -\partial_1 u_1 (x_1,x_2)
+2x_1\,\partial_1^2 u_1(x_1,x_2) - x_1^2 \partial_1^2 f(x_1,x_2),\ \  (x_1,x_2)\in\Omega^{+} 
, \label{eq:ext_u1}
\\\notag \\ 
\partial_1\widetilde u_2(-x_1,x_2)
&= - \partial_2 \widetilde u_1(-x_1,x_2) -\partial_2 u_1(x_1,x_2)-\partial_1 u_2(x_1,x_2) \nonumber \\ & \ \ +\frac{2\lambda c}{\mu}x_1 \partial_2 f(x_1,x_2) -4x_1 \partial_{12} u_1(x_1,x_2) \notag \\ & \ \ + 2x_1^2 \partial_{12} f(x_1,x_2), \ \ (x_1,x_2)\in\Omega^{+},
 \label{eq:ext_u2}
\\\notag\\
\partial_1\widetilde f(-x_1,x_2)
&=\partial_1 f(x_1,x_2) + \frac{2\lambda c}{\mu}\,\partial_1 f(x_1,x_2) \nonumber \\& \ \ -4\,\partial_1^2 u_1(x_1,x_2)
+2x_1\,\partial_1^2 f(x_1,x_2), \ \  (x_1,x_2)\in\Omega^{+}. \label{eq:ext_f}
\end{align}
To determine $\widetilde u_1$, $\widetilde u_2$, and $\widetilde f$, we integrate the
first-order system \eqref{eq:ext_u1}-\eqref{eq:ext_f} with respect to the variable $x_1$,
while keeping $x_2$ fixed.
This yields
\begin{align*}
\widetilde f(-x_1,x_2)
&= \widetilde f(0,x_2)
-\left(1+\frac{2\lambda c}{\mu}\right)
\bigl(f(x_1,x_2)-f(0,x_2)\bigr)
+4\bigl(\partial_1 u_1(x_1,x_2)-\partial_1 u_1(0,x_2)\bigr) \\
&\ \
-2\int_0^{x_1} s\,\partial_1^2 f(s,x_2)\,ds .
\end{align*}
Using integration by parts, the last term can be written as
$$
2\int_0^{x_1} s\,\partial_1^2 f(s,x_2)\,ds
=2x_1\,\partial_1 f(x_1,x_2)
-2\bigl(f(x_1,x_2)-f(0,x_2)\bigr).
$$
Substituting this identity into the previous expression, we obtain
\begin{align*}
\widetilde f(-x_1,x_2)
&= \widetilde f(0,x_2)
-\left(1+\frac{2\lambda c}{\mu}\right)\bigl(f(x_1,x_2)-f(0,x_2)\bigr)
+4\bigl(\partial_1 u_1(x_1,x_2)-\partial_1 u_1(0,x_2)\bigr) \\
&\ \
-2x_1\,\partial_1 f(x_1,x_2)+2( f(x_1,x_2)- f(0,x_2)).
\end{align*}
After combining similar terms, we get
\begin{align*}
\widetilde f(-x_1,x_2)
&= \widetilde f(0,x_2)
+\frac{\mu-\lambda}{\lambda+\mu}\bigl(f(x_1,x_2)-f(0,x_2)\bigr)
+4\bigl(\partial_1 u_1(x_1,x_2)-\partial_1 u_1(0,x_2)\bigr) \\
&\ \ 
-2x_1\,\partial_1 f(x_1,x_2).
\end{align*}
Since $\widetilde f$ is required to coincide with $f$ on $\gamma$, we have
$$
\widetilde f(0,x_2)=f(0,x_2).
$$
Consequently, the above expression simplifies to
\begin{align}
\widetilde f(-x_1,x_2)
&= \frac{2\lambda}{\lambda+\mu}\,f(0,x_2)+\frac{\mu-\lambda}{\lambda+\mu}\,f(x_1,x_2)
+
4\,\partial_1 u_1(x_1,x_2)
-4\,\partial_1 u_1(0,x_2) \nonumber\\
&\ \ 
-2x_1\,\partial_1 f(x_1,x_2).
\label{eq:tilde_f_formula}
\end{align}
Next, we substitute \eqref{eq:tilde_f_formula} into equation \eqref{eq:ext_u1}, we have
\begin{align}
\partial_1 \widetilde u_1(-x_1,x_2)
&= \frac{\lambda c}{2\mu} 
 \Big[\frac{2\lambda}{\lambda +\mu} f(0,x_2)+\frac{\mu-\lambda}{\lambda +\mu} f(x_1,x_2)
+4(\partial_1 u_1(x_1,x_2) 
- \partial_1 u_1(0,x_2))\notag \\
& \ \ - 2x_1 \partial_1 f(x_1,x_2) \Big]+\frac{\lambda c}{2\mu} f(x_1,x_2)-\frac{\lambda c}{\mu} x_1 \partial_1 f(x_1,x_2)-\partial_1 u_1(x_1,x_2) \notag \\ 
& \ \ +2x_1\,\partial_1^2 u_1(x_1,x_2)- x_1^2 \partial_1^2 f(x_1,x_2) \notag
\end{align}
After combining similar terms, we get
\begin{align} \label{eq 39}
\partial_1 \widetilde u_1(-x_1,x_2)
&= \left(\frac{\lambda}{\lambda + \mu}\right)^2f(0,x_2)+\frac{\mu\lambda}{(\lambda + \mu)^2} f(x_1,x_2)  +\frac{\lambda - \mu}{\lambda +\mu}
\partial_1 u_1(x_1,x_2) \notag \\
& \ \ - \frac{2\lambda}{\lambda + \mu}\,
\partial_1 u_1(0,x_2) -  \frac{2\lambda}{\lambda + \mu}\,
x_1\partial_1 f(x_1,x_2) +2x_1\,\partial_1^2 u_1(x_1,x_2)- x_1^2\,\partial_1^2 f(x_1,x_2).
\end{align}
Integrating \eqref{eq 39} with respect to $x_1$ from $0$ to $x_1$, we arrive at
the explicit representation
\begin{equation}\label{eq:ext_u1_integrated}
\begin{aligned}
\widetilde u_1(-x_1,x_2)
&= \widetilde u_1(0,x_2)- \left(\frac{\lambda}{\lambda+\mu}\right)^2 x_1f(0,x_2)\,
- \frac{\mu\lambda}{(\lambda+\mu)^2}
\int_0^{x_1} f(s,x_2)\,ds
\\
&\ 
-\frac{\lambda-\mu}{\lambda+\mu}
\big(u_1(x_1,x_2)-u_1(0,x_2)\big)
+ \frac{2\lambda}{\lambda+\mu}\,x_1\partial_1 u_1(0,x_2)
\\
&\ 
 + \frac{2\lambda}{\lambda + \mu} \int_0^{x_1} s\,\partial_1 f(s,x_2)\,ds- 2\int_0^{x_1} s\,\partial_1^2 u_1(s,x_2)\,ds + \int_0^{x_1} s^2\,\partial_1^2 f(s,x_2)\,ds
 .
\end{aligned}
\end{equation}

Using the integration by parts, we have
$$
\hspace{-1.9cm}
\int_0^{x_1} s\,\partial_1 f(s,x_2)\,ds
= x_1 f(x_1,x_2)
-\int_0^{x_1} f(s,x_2)ds.
$$
$$
\hspace{-1.9cm}
\int_0^{x_1} s\,\partial_1^2 u_1(s,x_2)\,ds
= x_1\partial_1 u_1(x_1,x_2)
- u_1(x_1,x_2)
+ u_1(0,x_2).
$$
and
$$
\int_0^{x_1} s^2\,\partial_1^2 f(s,x_2)\,ds
= x_1^2\,\partial_1 f(x_1,x_2)
-2x_1 f(x_1,x_2)
+2\int_0^{x_1} f(s,x_2)\,ds,
$$
Substituting these identities into \eqref{eq:ext_u1_integrated}, we get
\begin{equation}
\label{abcd}
\begin{aligned}
\widetilde u_1(-x_1,x_2)
&= \widetilde u_1(0,x_2) - \left(\frac{\lambda}{\lambda+\mu}\right)^2 x_1 f(0,x_2)
+ \frac{\mu(\lambda + 2\mu)}{(\lambda+\mu)^2}
\int_0^{x_1} f(s,x_2)\,ds
\\
&\ \ 
+ \frac{\lambda+3\mu}{\lambda+\mu}
\big(u_1(x_1,x_2)-u_1(0,x_2)\big)
+ \frac{2\lambda}{\lambda+\mu}\,x_1\partial_1 u_1(0,x_2)
\\
&\ \ - 2x_1\partial_1 u_1(x_1,x_2) - \frac{2\mu}{\lambda+\mu}x_1 f(x_1,x_2) + x_1^2\,\partial_1 f(x_1,x_2).
\end{aligned}
\end{equation}
Since then, we have
$$
\widetilde u_1(0,x_2)=u_1(0,x_2).
$$
Moreover, from the boundary condition $v_1=0$ on $\gamma$ and the identity
$
v_1=2\mu\,\partial_1u_1-\lambda c\,f,
$
it follows that
$$
2\mu\,\partial_1u_1(0,x_2)
=
\lambda c\,f(0,x_2).
$$
Since $c=\mu/(\lambda+\mu)$, we obtain
\begin{equation} \label{partial u_1}
    \partial_1u_1(0,x_2)
=
\frac{\lambda}{2(\lambda+\mu)}\,f(0,x_2)   
\end{equation}
Substituting these relations into \eqref{abcd}, we arrive at
\begin{equation}\label{eq:tilde_u1_formula}
\begin{aligned}
\widetilde u_1(-x_1,x_2)
&=
\frac{\lambda+3\mu}{\lambda+\mu}\,
u_1(x_1,x_2)
-\frac{2\mu}{\lambda+\mu}\,
u_1(0,x_2)
\\
&\quad
+\frac{\mu(\lambda+2\mu)}{(\lambda+\mu)^2}
\int_0^{x_1} f(s,x_2)\,ds
\\
&\quad
-2x_1\,\partial_1u_1(x_1,x_2)
-\frac{2\mu}{\lambda+\mu}\,
x_1f(x_1,x_2)
+x_1^2\,\partial_1f(x_1,x_2).
\end{aligned}
\end{equation}
This proves Equation \eqref{eq:2.5}. 
Next we determine $\widetilde{u}_2(-x_1,x_2)$ . First, we compute $\partial_2 \widetilde{u}_1(-x_1,x_2)$ from \eqref{eq:tilde_u1_formula}, 
 Differentiating with respect to $x_2$. Since $x_1$ is fixed while differentiating in $x_2$, we get
 \begin{equation}\label{eq:ok}
\begin{aligned}
\partial_2 \widetilde{u}_1(-x_1,x_2)
&= \frac{\lambda+3\mu}{\lambda+\mu}\,
\partial_2 u_1(x_1,x_2) -\frac{2\mu}{\lambda+\mu}\,
\partial_2u_1(0,x_2)
+\frac{\mu(\lambda+2\mu)}{(\lambda+\mu)^2}
\int_0^{x_1} \partial_2f(s,x_2)\,ds
\\
&\ 
- 2x_1\partial_{12} u_1(x_1,x_2) -\frac{2\mu}{\lambda+\mu}\,
x_1 \partial_2 f(x_1,x_2)+ x_1^2\,\partial_{12} f(x_1,x_2).
\end{aligned}
\end{equation}
Substituting \eqref{eq:ok} into \eqref{eq:ext_u2} 
we obtain 
\begin{align*}
\partial_1 \widetilde{u}_2(-x_1,x_2)
&= -\frac{\lambda+3\mu}{\lambda+\mu}\,\partial_2 u_1(x_1,x_2) +\frac{2\mu}{\lambda+\mu}\,\partial_2u_1(0,x_2)
-\frac{\mu(\lambda+2\mu)}{(\lambda+\mu)^2}\,
\int_0^{x_1} \partial_2f(s,x_2)\,ds
\\
&\ \ 
+ 2x_1\partial_{12} u_1(x_1,x_2) +\frac{2\mu}{\lambda+\mu}\,
x_1 \partial_2 f(x_1,x_2)- x_1^2\partial_{12} f(x_1,x_2)\\
&\ \
- \partial_2 u_1(x_1,x_2)-\partial_1 u_2(x_1,x_2)
+\frac{2\lambda }{\lambda+\mu}x_1\,\partial_2 f(x_1,x_2)\\ & \ \
-4x_1\,\partial_{12} u_1(x_1,x_2)+ 2x_1^2\,\partial_{12} f(x_1,x_2).
\end{align*}
After combining similar terms, we get
\begin{align}\label{eq:ext_u2_rearranged}
\partial_1 \widetilde u_2(-x_1,x_2)
&= -\frac{2(\lambda+2\mu)}{\lambda+\mu}\,\partial_2 u_1(x_1,x_2) +\frac{2\mu}{\lambda+\mu}\,\partial_2u_1(0,x_2)-\partial_1u_2(x_1,x_2) \notag   \\ & \ \
-\frac{\mu(\lambda+2\mu)}{(\lambda+\mu)^2}\,
\int_0^{x_1} \partial_2f(s,x_2)\,ds
- 2x_1\partial_{12} u_1(x_1,x_2) +2
x_1 \partial_2 f(x_1,x_2) \notag \\ & \ \ + x_1^2\partial_{12} f(x_1,x_2).
\end{align}

Integrating the above identity with respect to $x_1$ from $0$ to $x_1$, and using
$$
\widetilde u_2(0,x_2)=u_2(0,x_2),
$$
we obtain
\begin{align*}
\widetilde u_2(-x_1,x_2)
&=
u_2(0,x_2)
+\frac{2(\lambda+2\mu)}{\lambda+\mu}
\int_0^{x_1}\partial_2u_1(s,x_2)\,ds
-\frac{2\mu}{\lambda+\mu}\,
x_1\partial_2u_1(0,x_2)\\\\ & \ \
+\int_0^{x_1}\partial_1u_2(s,x_2)\,ds
+\frac{\mu(\lambda+2\mu)}{(\lambda+\mu)^2}
\int_0^{x_1}\int_0^s
\partial_2f(t,x_2)\,dt\,ds
\\\\
&\ \
+2\int_0^{x_1}s\,\partial_{12}u_1(s,x_2)\,ds
-2\int_0^{x_1}s\,\partial_2f(s,x_2)\,ds
-\int_0^{x_1}s^2\,\partial_{12}f(s,x_2)\,ds.
\end{align*}
Using integration by parts, we have
$$
\int_0^{x_1}\partial_1u_2(s,x_2)\,ds
=
u_2(x_1,x_2)-u_2(0,x_2),
$$\\
$$
2\int_0^{x_1}s\,\partial_{12}u_1(s,x_2)\,ds
=
2x_1\partial_2u_1(x_1,x_2)
-2\int_0^{x_1}\partial_2u_1(s,x_2)\,ds,
$$\\
and
$$
\int_0^{x_1}s^2\,\partial_{12}f(s,x_2)\,ds
=
x_1^2\partial_2f(x_1,x_2)
-2\int_0^{x_1}s\,\partial_2f(s,x_2)\,ds,
$$\\
we arrive at
\begin{equation}\label{eq:tilde_u2_formula}
\begin{aligned}
\widetilde u_2(-x_1,x_2)
&=
u_2(x_1,x_2)
+\frac{2\mu}{\lambda+\mu}
\int_0^{x_1}\partial_2u_1(s,x_2)\,ds
-\frac{2\mu}{\lambda+\mu}\,
x_1\partial_2u_1(0,x_2) \\ & \ \
+2x_1\partial_2u_1(x_1,x_2)
+\frac{\mu(\lambda+2\mu)}{(\lambda+\mu)^2}
\int_0^{x_1}(x_1-t)\,
\partial_2f(t,x_2)\,dt
\\
&\quad
-x_1^2\partial_2f(x_1,x_2).
\end{aligned}
\end{equation}

  \medskip
\noindent \textbf{Step 5 (Regularity of the continuation across $\gamma$).} We next verify that the functions \eqref{eq:tilde_f_formula}, \eqref{eq:tilde_u1_formula} and \eqref{eq:tilde_u2_formula} indeed define a smooth continuation of the original field across $\gamma$. By the choice of the integration constant in the construction above, the extended field satisfies
$$
\widetilde u=u,\qquad \widetilde f=f
\quad\text{on }\gamma.
$$ Equivalently, these identities can be checked directly by setting $x_1=0$ in
\eqref{eq:tilde_f_formula}, \eqref{eq:tilde_u1_formula} and \eqref{eq:tilde_u2_formula}.
We next verify that the first-order derivatives also agree across
$\gamma$. 

Now, letting $x_1\to0^+$ in \eqref{eq 39}, we obtain
$$
\begin{aligned}
(\partial_1 \widetilde u_1)(0,x_2)
&=
\left(\frac{\lambda}{\lambda+\mu}\right)^2f(0,x_2)
+\frac{\mu\lambda}{(\lambda+\mu)^2}f(0,x_2) \\
&\quad
+\frac{\lambda-\mu}{\lambda+\mu}
\partial_1u_1(0,x_2)
-\frac{2\lambda}{\lambda+\mu}
\partial_1u_1(0,x_2).
\end{aligned}
$$
Combining the first two terms and the last two terms gives
$$
(\partial_1 \widetilde u_1)(0,x_2)
=
\frac{\lambda}{\lambda+\mu}f(0,x_2)
-\partial_1u_1(0,x_2).
$$
Using \eqref{partial u_1}, we have
$$
\frac{\lambda}{\lambda+\mu}f(0,x_2)
=
2\partial_1u_1(0,x_2),
$$
and hence
\begin{equation} \label{del_1 tildau_1}
(\partial_1 \widetilde u_1)(0,x_2)
=
\partial_1u_1(0,x_2).
\end{equation}
For the second component, letting $x_1\to0^+$ in \eqref{eq:ext_u2_rearranged}, we get
$$
\begin{aligned}
(\partial_1 \widetilde u_2)(0,x_2)
&=
-\frac{2(\lambda+2\mu)}{\lambda+\mu}
\partial_2u_1(0,x_2)
+\frac{2\mu}{\lambda+\mu}
\partial_2u_1(0,x_2) 
-\partial_1u_2(0,x_2) \\
&=
-2\partial_2u_1(0,x_2)
-\partial_1u_2(0,x_2).
\end{aligned}
$$
On the other hand, the second component of the traction-free boundary
condition gives
$$
v_2
=
\mu\big(\partial_2u_1+\partial_1u_2\big)
=
0
\qquad\text{on }\gamma.
$$
Since $\mu>0$, it follows that
$$
\partial_2u_1(0,x_2)
=
-\partial_1u_2(0,x_2).
$$
Consequently,
\begin{equation} \label{del_1 tildau_2}
(\partial_1 \widetilde u_2)(0,x_2)
=
\partial_1u_2(0,x_2).
\end{equation}
We next verify the corresponding identity for $\widetilde f$. Since
$\nabla\cdot v=0$ in $\Omega^+$ and $v_2=0$ on $\gamma$, we have
$$
\partial_1v_1(0,x_2)+\partial_2v_2(0,x_2)=0.
$$
As $v_2(0,x_2)=0$ on $\gamma$, differentiation in the
tangential direction yields
$
\partial_2v_2(0,x_2)=0,
$
and therefore
$
\partial_1v_1(0,x_2)=0.
$
Using
$$
v_1(x_1,x_2)=2\mu\partial_1u_1(x_1,x_2)-\lambda c f(x_1,x_2)\ in \ \Omega^+,
$$
we obtain
$$
2\mu\partial_{11}u_1(0,x_2)
=
\lambda c\,\partial_1f(0,x_2),
$$
that is,
\begin{equation}   \label{del11}
\partial_{11}u_1(0,x_2)
=
\frac{\lambda}{2(\lambda+\mu)}
\partial_1f(0,x_2).
\end{equation}
Now, letting $x_1\to0^+$ in \eqref{eq:ext_f}, we have
$$
\begin{aligned}
(\partial_1 \widetilde f)(0,x_2)
&=
\left(1+\frac{2\lambda c}{\mu}\right)
\partial_1f(0,x_2)
-4\partial_{11}u_1(0,x_2).
\end{aligned}
$$
Since
$
c=\mu/(\lambda+\mu),
$
and using \eqref{del11}, it follows that

\begin{equation}\label{del_1 tildaf}
\begin{aligned}
(\partial_1\widetilde f)(0,x_2)
&=
\left(1+\frac{2\lambda}{\lambda+\mu}\right)
\partial_1f(0,x_2)
-\frac{2\lambda}{\lambda+\mu}
\partial_1f(0,x_2) \\
&=
\partial_1f(0,x_2).
\end{aligned}
\end{equation}
The tangential derivatives agree automatically. Indeed, since
$
\widetilde u_j(0,x_2)=u_j(0,x_2),
\ j=1,2,
$
and
$
\widetilde f(0,x_2)=f(0,x_2),
$
as functions of the tangential variable $x_2$, differentiation with
respect to $x_2$ gives
$
\partial_2\widetilde u_j(0,x_2)
=
\partial_2u_j(0,x_2),
\ j=1,2,
$
and
$
\partial_2\widetilde f(0,x_2)
=
\partial_2f(0,x_2).
$
Combining these identities with \eqref{del_1 tildau_1}, \eqref{del_1 tildau_2}, and \eqref{del_1 tildaf}, we conclude that
$$
\widetilde u\in C^1(\Omega;\mathbb{R}^2),
\qquad
\widetilde f\in C^1(\Omega).
$$

\medskip
\noindent \textbf{Step 6 (The continuation solves the Lam\'e system in $\Omega$).} It remains to verify that the extension field $\widetilde u$  satisfies the Lam\'e system in the whole domain $\Omega$. 
Since $\widetilde u(x_1, x_2)=u(x_1,x_2)$ and $\widetilde f(x_1,x_2)=f(x_1,x_2)$ in $\Omega^+$ by construction,
it suffices to show that the pair $(\widetilde u(x_1,x_2),\widetilde f(x_1,x_2))$ satisfies
\begin{equation}\label{eq:duffin_extended}
\Delta \widetilde u(x_1,x_2)=\nabla \widetilde f(x_1,x_2),
\ \
\nabla\cdot \widetilde u(x_1,x_2)=-c \widetilde f(x_1,x_2),
\ \
\Delta \widetilde f(x_1,x_2)=0
\quad \mbox{in }\Omega,
\end{equation}
where $c=\mu/(\lambda+\mu)$.
For $(x_1,x_2)\in\Omega^+$, define
$$
F(x_1,x_2):= \widetilde f(-x_1,x_2),
$$
and
$$
U_j(x_1,x_2):= \widetilde u_j(-x_1,x_2),
\ \ j=1,2.
$$
Then proving \eqref{eq:duffin_extended} in $\Omega^-$ is equivalent to
showing that
\begin{equation}\label{eq:verification_system}
\begin{cases}
\partial_{11}U_1(x_1,x_2) +\partial_{22}U_1(x_1,x_2)+\partial_1F(x_1,x_2)=0, \ \ (x_1,x_2)\in\Omega^+,\\\\
\partial_{11}U_2(x_1,x_2)+\partial_{22}U_2(x_1,x_2)-\partial_2F(x_1,x_2)=0,\ \ (x_1,x_2)\in\Omega^+, \\\\
-\partial_1U_1(x_1,x_2)+\partial_2U_2(x_1,x_2)+cF(x_1,x_2)=0,\ \ (x_1,x_2)\in\Omega^+,\\\\
\partial_{11}F(x_1,x_2)+\partial_{22}F(x_1,x_2)=0, \ \ (x_1,x_2)\in\Omega^+.
\end{cases}
\end{equation}
For convenience, set
$$
A:=\frac{\lambda+3\mu}{\lambda+\mu},
\ \
B:=\frac{2\mu}{\lambda+\mu},
\ \
C:=\frac{\mu(\lambda+2\mu)}{(\lambda+\mu)^2}.
$$
Using \eqref{eq:tilde_f_formula}, \eqref{eq:tilde_u1_formula} and \eqref{eq:tilde_u2_formula}, we have

\begin{align*}
F(x_1,x_2)
&=
\frac{2\lambda}{\lambda+\mu}f(0,x_2)
+\frac{\mu-\lambda}{\lambda+\mu}f(x_1,x_2)
+4\partial_1u_1(x_1,x_2)\\
&\ \
-4\partial_1u_1(0,x_2)
-2x_1\partial_1f(x_1,x_2), \ \ (x_1,x_2) \in \O^+.
\end{align*}

\begin{align*}
U_1(x_1,x_2)
&=
Au_1(x_1,x_2)
-Bu_1(0,x_2)
+C\int_0^{x_1}f(s,x_2)ds
\\
&\ \
-2x_1\partial_1u_1(x_1,x_2)
-Bx_1f(x_1,x_2)
+x_1^2\partial_1f(x_1,x_2), \ \ (x_1,x_2) \in \O^+.
\end{align*}
and
$$
\begin{aligned}
U_2(x_1,x_2)
&=
u_2(x_1,x_2)
+B\int_0^{x_1}\partial_2u_1(s,x_2)\,ds
-Bx_1\partial_2u_1(0,x_2)
+2x_1\partial_2u_1(x_1,x_2)
\\
&\ \
+C\int_0^{x_1}(x_1-t)\partial_2f(t,x_2)\,dt
-x_1^2\partial_2f, \ \ (x_1,x_2) \in \O^+.
\end{aligned}
$$
We first verify the first identity in
\eqref{eq:verification_system}, that is
$$
\partial_{11}U_1(x_1,x_2)+\partial_{22}U_1(x_1,x_2)+\partial_1F(x_1,x_2)=0, \ \ (x_1,x_2)\in \Omega^+.
$$
Using $\Delta u_1(x_1,x_2)=\partial_1f(x_1,x_2)$ and $\Delta f(x_1,x_2)=0$, $(x_1,x_2)\in \Omega^+$, we compute term by term:
$$
\Delta(u_1(x_1,x_2))=\partial_1f(x_1,x_2),
$$
$$
\Delta(u_1(0,x_2))=\partial_{22}u_1(0,x_2),
$$
and
$$
\Delta\left(\int_0^{x_1}f(s,x_2)\,ds\right)
=
\partial_1f(0,x_2).
$$
Furthermore,
$$
\Delta(2x_1\partial_1u_1(x_1,x_2))
=
4\partial_{11}u_1(x_1,x_2)+2x_1\partial_{11}f(x_1,x_2),
$$

$$
\Delta(x_1f(x_1,x_2))
=
2\partial_1f(x_1,x_2),
$$
and
$$
\Delta(x_1^2\partial_1f(x_1,x_2))
=
2\partial_1f(x_1,x_2)+4x_1\partial_{11}f(x_1,x_2).
$$
Also,
$$
\partial_1F(x_1,x_2)
=
\frac{\mu-\lambda}{\lambda+\mu}\partial_1f(x_1,x_2)
+4\partial_{11}u_1(x_1,x_2)
-2\partial_1f(x_1,x_2)
-2x_1\partial_{11}f(x_1,x_2).
$$
Adding these identities gives
$$
\partial_{11}U_1(x_1,x_2)+\partial_{22}U_1(x_1,x_2)+\partial_1F(x_1,x_2)
=
-B\partial_{22}u_1(0,x_2)
+
C\partial_1f(0,x_2),
$$
because all the terms containing
$$
\partial_1f(x_1,x_2),\ \ \
\partial_{11}u_1(x_1,x_2),\ \ \
x_1\partial_{11}f(x_1,x_2)
$$
cancel. Therefore,
\begin{equation}\label{eq:verify}
\partial_{11}U_1(x_1,x_2)+\partial_{22}U_1(x_1,x_2)+\partial_1F(x_1,x_2)
=
-\frac{2\mu}{\lambda+\mu}\partial_{22}u_1(0,x_2)
+
\frac{\mu(\lambda+2\mu)}{(\lambda+\mu)^2}
\partial_1f(0,x_2).
\end{equation}
We first see some boundary identities which will be used in the verification below. Recall from \eqref{partial u_1} that
\begin{equation*}\label{eq:bd_id_1}
\partial_1u_1(0,x_2)
=
\frac{\lambda}{2(\lambda+\mu)}f(0,x_2).
\end{equation*}
Differentiating this identity tangentially with respect to $x_2$, we obtain
\begin{equation}\label{eq:bd_id_2}
\partial_{12}u_1(0,x_2)
=
\frac{\lambda}{2(\lambda+\mu)}\partial_2f(0,x_2),
\end{equation}
and differentiating again with respect to $x_2$,
\begin{equation}\label{eq:bd_id_3}
\partial_{122}u_1(0,x_2)
=
\frac{\lambda}{2(\lambda+\mu)}\partial_{22}f(0,x_2).
\end{equation}
Moreover, from \eqref{del11}, we have 
\begin{equation}\label{eq:bd_id_4}
\partial_{11}u_1(0,x_2)
=
\frac{\lambda}{2(\lambda+\mu)}\partial_1f(0,x_2).
\end{equation}
Since $\Delta u_1(x_1,x_2)=\partial_1f(x_1,x_2)$ in $\Omega^+$, evaluating on $\gamma$ gives
$$
\partial_{11}u_1(0,x_2)+\partial_{22}u_1(0,x_2)
=
\partial_1f(0,x_2).
$$
Combining this with \eqref{eq:bd_id_4}, we obtain
\begin{equation}\label{eq:bd_id_5}
\partial_{22}u_1(0,x_2)
=
\frac{\lambda+2\mu}{2(\lambda+\mu)}
\partial_1f(0,x_2).
\end{equation}
Differentiating \eqref{eq:bd_id_4} tangentially with respect to $x_2$, we get
$$
\partial_{112}u_1(0,x_2)
=
\frac{\lambda}{2(\lambda+\mu)}
\partial_{12}f(0,x_2).
$$
Also, differentiating $\Delta u_1(x_1,x_2)=\partial_1f(x_1,x_2)$ with respect to $x_2$ gives
$$
\partial_{112}u_1(x_1,x_2)+\partial_{222}u_1(x_1,x_2)=\partial_{12}f(x_1,x_2).
$$
Hence,
\begin{equation}\label{eq:bd_id_6}
\partial_{222}u_1(0,x_2)
=
\frac{\lambda+2\mu}{2(\lambda+\mu)}
\partial_{12}f(0,x_2).
\end{equation}
Now we can use \eqref{eq:bd_id_5} into \eqref{eq:verify}, the right-hand side becomes zero. Hence,
\begin{equation}\label{eq:verify_u1}
\partial_{11}U_1(x_1,x_2)+\partial_{22}U_1(x_1,x_2)+\partial_1F(x_1,x_2)=0, \ \ (x_1,x_2)\in \Omega^+
\end{equation}
Next, we verify that $F$ is harmonic in $\Omega^+$. We compute
$$
\partial_{11}F(x_1,x_2)
=
\frac{\mu-\lambda}{\lambda+\mu}\partial_{11}f(x_1,x_2)
+4\partial_{111}u_1(x_1,x_2)
-4\partial_{11}f(x_1,x_2)
-2x_1\partial_{111}f(x_1,x_2),
$$
and
\begin{align*}
\partial_{22}F(x_1,x_2)
&=
\frac{2\lambda}{\lambda+\mu}\partial_{22}f(0,x_2)
+\frac{\mu-\lambda}{\lambda+\mu}\partial_{22}f(x_1,x_2)
+4\partial_{122}u_1(x_1,x_2)
\\ &\ \ -4\partial_{122}u_1(0,x_2)
-2x_1\partial_{122}f(x_1,x_2).
\end{align*}
Adding these two identities and using $\Delta f(x_1,x_2)=0, \ \ (x_1,x_2) \in \Omega^+$ and
$$
\partial_{111}u_1(x_1,x_2)+\partial_{122}u_1(x_1,x_2)
=
\partial_1(\Delta u_1(x_1,x_2))
=
\partial_{11}f(x_1,x_2),
$$
we obtain
$$
\Delta F(x_1,x_2)
=
\frac{2\lambda}{\lambda+\mu}\partial_{22}f(0,x_2)
-4\partial_{122}u_1(0,x_2).
$$
By \eqref{eq:bd_id_3},
$$
\partial_{122}u_1(0,x_2)
=
\frac{\lambda}{2(\lambda+\mu)}
\partial_{22}f(0,x_2),
$$
Therefore
$$
\Delta F(x_1,x_2)=0, \ \ (x_1,x_2)\in \Omega^+,
$$
Thus,
\begin{equation}\label{eq:verify_f}
\partial_{11}F(x_1,x_2)+\partial_{22}F(x_1,x_2)=0, \ \ (x_1,x_2)\in \Omega^+.
\end{equation}
Next, we verify
$$
\partial_{11}U_2(x_1,x_2)+\partial_{22}U_2(x_1,x_2)-\partial_2F(x_1,x_2)=0.
$$
Using $\Delta u_2(x_1,x_2)=\partial_2f(x_1,x_2)$, the contribution from the first term of $U_2(x_1,x_2)$ is
$$
\Delta u_2(x_1,x_2)=\partial_2f(x_1,x_2).
$$
For the integral term, we compute
$$
\Delta\left(B\int_0^{x_1}\partial_2u_1(s,x_2)\,ds\right)
=
B\left[
\partial_2f(x_1,x_2)-\partial_2f(0,x_2)+\partial_{12}u_1(0,x_2)
\right].
$$
Also,
$$
\Delta\left(-Bx_1\partial_2u_1(0,x_2)\right)
=
-Bx_1\partial_{222}u_1(0,x_2),
$$
and
$$
\Delta(2x_1\partial_2u_1(x_1,x_2))
=
4\partial_{12}u_1(x_1,x_2)+2x_1\partial_{12}f(x_1,x_2).
$$
Furthermore,
$$
\Delta\left(
C\int_0^{x_1}(x_1-t)\partial_2f(t,x_2)\,dt
\right)
=
C\partial_2f(0,x_2)+Cx_1\partial_{12}f(0,x_2),
$$
and
$$
\Delta(-x_1^2\partial_2f(x_1,x_2))
=
-2\partial_2f(x_1,x_2)-4x_1\partial_{12}f(x_1,x_2).
$$
Consequently,
\begin{align*}
\Delta U_2(x_1,x_2)
&=
\partial_2f(x_1,x_2)
+B\left[
\partial_2f(x_1,x_2)-\partial_2f(0,x_2)+\partial_{12}u_1(0,x_2)
\right]
-Bx_1\partial_{222}u_1(0,x_2)
\\
&\ \
+4\partial_{12}u_1(x_1,x_2)+2x_1\partial_{12}f(x_1,x_2)
+C\partial_2f(0,x_2)+Cx_1\partial_{12}f(0,x_2)\\& \ \ 
-2\partial_2f(x_1,x_2)-4x_1\partial_{12}f(x_1,x_2).
\end{align*}
On the other hand,
$$
\partial_2F(x_1,x_2)
=
\frac{2\lambda}{\lambda+\mu}\partial_2f(0,x_2)
+\frac{\mu-\lambda}{\lambda+\mu}\partial_2f(x_1,x_2)
+4\partial_{12}u_1(x_1,x_2)
-4\partial_{12}u_1(0,x_2)
-2x_1\partial_{12}f(x_1,x_2).
$$
Subtracting, all interior terms cancel, and we obtain
\begin{align*}
\Delta U_2(x_1,x_2)-\partial_2F(x_1,x_2)
&=
\left(B+4\right)\partial_{12}u_1(0,x_2)
+
\left(-B+C-\frac{2\lambda}{\lambda+\mu}\right)
\partial_2f(0,x_2)
\\
&\quad
+x_1\left[
-B\partial_{222}u_1(0,x_2)
+C\partial_{12}f(0,x_2)
\right].
\end{align*}
Substituting the values of $B$ and $C$, this becomes
\begin{align*}
\Delta U_2(x_1,x_2)-\partial_2F(x_1,x_2)
&=
\frac{2(2\lambda+3\mu)}{\lambda+\mu}
\partial_{12}u_1(0,x_2)
-\frac{\lambda(2\lambda+3\mu)}{(\lambda+\mu)^2}
\partial_2f(0,x_2)
\\
&\quad
+x_1\left[
-\frac{2\mu}{\lambda+\mu}\partial_{222}u_1(0,x_2)
+
\frac{\mu(\lambda+2\mu)}{(\lambda+\mu)^2}
\partial_{12}f(0,x_2)
\right].
\end{align*}
Using \eqref{eq:bd_id_2}, the first line vanishes. Using
\eqref{eq:bd_id_6}, the second line also vanishes. Therefore
\begin{equation}\label{eq:verify_u2}
\partial_{11}U_2(x_1,x_2)+\partial_{22}U_2(x_1,x_2)-\partial_2F(x_1,x_2)=0.
\end{equation}

Finally, we verify the divergence identity
$$
-\partial_1U_1(x_1,x_2)+\partial_2U_2(x_1,x_2)+cF(x_1,x_2)=0.
$$
First,
\begin{align*}
\partial_1U_1(x_1,x_2)
&=
A\partial_1u_1(x_1,x_2)
+Cf(x_1,x_2)
-2\partial_1u_1(x_1,x_2)
-2x_1\partial_{11}u_1(x_1,x_2)\\
&\ \
-Bf(x_1,x_2)
-Bx_1\partial_1f(x_1,x_2)
+2x_1\partial_1f(x_1,x_2)
+x_1^2\partial_{11}f(x_1,x_2).
\end{align*}
Thus
\begin{align*}
-\partial_1U_1(x_1,x_2)
&=
-A\partial_1u_1(x_1,x_2)
-Cf(x_1,x_2)
+2\partial_1u_1(x_1,x_2)
+2x_1\partial_{11}u_1(x_1,x_2)
+Bf(x_1,x_2) \\
&\ \
+Bx_1\partial_1f(x_1,x_2)
-2x_1\partial_1f(x_1,x_2)
-x_1^2\partial_{11}f(x_1,x_2).
\end{align*}
Next,
\begin{align*}
\partial_2U_2(x_1,x_2)
&=
\partial_2u_2(x_1,x_2)
+B\int_0^{x_1}\partial_{22}u_1(s,x_2)\,ds
-Bx_1\partial_{22}u_1(0,x_2)
+2x_1\partial_{22}u_1(x_1,x_2)
\\
&\quad
+C\int_0^{x_1}(x_1-t)\partial_{22}f(t,x_2)\,dt
-x_1^2\partial_{22}f.
\end{align*}
Using
\begin{align*}
\partial_2u_2(x_1,x_2)&=-cf-\partial_1u_1(x_1,x_2),
\\
\partial_{22}u_1(x_1,x_2)&=\partial_1f(x_1,x_2)-\partial_{11}u_1(x_1,x_2),
\\
\partial_{22}f(x_1,x_2)&=-\partial_{11}f(x_1,x_2),
\end{align*}
we obtain
$$
\int_0^{x_1}\partial_{22}u_1(s,x_2)\,ds
=
f(x_1,x_2)-f(0,x_2)
-\partial_1u_1(x_1,x_2)
+\partial_1u_1(0,x_2),
$$
and
$$
\int_0^{x_1}(x_1-t)\partial_{22}f(t,x_2)\,dt
=
-f(x_1,x_2)+f(0,x_2)+x_1\partial_1f(0,x_2).
$$
Therefore
\begin{align*}
\partial_2U_2(x_1,x_2)
&=
-cf-\partial_1u_1(x_1,x_2)
+B\left[
f(x_1,x_2)-f(0,x_2)-\partial_1u_1(x_1,x_2)+\partial_1u_1(0,x_2)
\right]
\\
&\quad
-Bx_1\partial_{22}u_1(0,x_2)
+2x_1\partial_1f(x_1,x_2)
-2x_1\partial_{11}u_1(x_1,x_2)
\\
&\quad
+C\left[
-f(x_1,x_2)+f(0,x_2)+x_1\partial_1f(0,x_2)
\right]
+x_1^2\partial_{11}f(x_1,x_2).
\end{align*}
Also,
$$
cF(x_1,x_2)=
\frac{2\lambda c}{\lambda+\mu}f(0,x_2)
+
\frac{c(\mu-\lambda)}{\lambda+\mu}f(x_1,x_2)
+
4c\partial_1u_1(x_1,x_2)
-4c\partial_1u_1(0,x_2)
-2cx_1\partial_1f(x_1,x_2).
$$
Now adding $-\partial_1U_1(x_1,x_2)+\partial_2U_2(x_1,x_2)+cF(x_1,x_2)$, all interior terms cancel. Indeed, the coefficient of
$\partial_1u_1(x_1,x_2)$ is
$$
-A+2-1-B+4c=0;
$$
the coefficient of $f$ is
$$
-2C+2B-c+\frac{c(\mu-\lambda)}{\lambda+\mu}=0;
$$
the coefficient of $x_1\partial_1f(x_1,x_2)$ is
$$
B-2c=0;
$$
and the terms containing $x_1\partial_{11}u_1(x_1,x_2)$ and
$x_1^2\partial_{11}f(x_1,x_2)$ cancel directly. Therefore only boundary terms remain:
\begin{align*}
-\partial_1U_1(x_1,x_2)+\partial_2U_2(x_1,x_2)+cF(x_1,x_2)
&=
\frac{\lambda\mu}{(\lambda+\mu)^2}f(0,x_2)
-\frac{2\mu}{\lambda+\mu}\partial_1u_1(0,x_2)
\\
&\quad
+x_1\left[
-\frac{2\mu}{\lambda+\mu}\partial_{22}u_1(0,x_2)
+
\frac{\mu(\lambda+2\mu)}{(\lambda+\mu)^2}
\partial_1f(0,x_2)
\right].
\end{align*}
Using \eqref{eq:bd_id_1} and \eqref{eq:bd_id_5}, the right-hand side vanishes. Hence
\begin{equation}\label{eq:verify_div}
-\partial_1U_1(x_1,x_2)+\partial_2U_2(x_1,x_2)+cF(x_1,x_2)=0.
\end{equation}
Combining \eqref{eq:verify_u1}, \eqref{eq:verify_f}, \eqref{eq:verify_u2}
and \eqref{eq:verify_div}, we have proved \eqref{eq:verification_system}. Hence
$$
\Delta \widetilde u(x_1,x_2)=\nabla \widetilde f(x_1,x_2),\qquad
\nabla\cdot \widetilde u(x_1,x_2)=-c\widetilde f(x_1,x_2),\qquad
\Delta \widetilde f(x_1,x_2)=0
\quad\text{in }\Omega^-.
$$
On the other hand, since $\widetilde u(x_1,x_2)=u(x_1,x_2)$ and $\widetilde f(x_1,x_2)=f(x_1,x_2)$ in $\Omega^+\cup\gamma$, the pair $(\widetilde u(x_1,x_2),\widetilde f(x_1,x_2))$ gives the desired analytic continuation of $(u,f)$ across $\gamma$. Consequently,
$$
\mu\Delta \widetilde u(x_1,x_2)+(\lambda+\mu)\nabla(\nabla\cdot \widetilde u(x_1,x_2))
=
\mu\nabla \widetilde f(x_1,x_2)-(\lambda+\mu)c\nabla \widetilde f(x_1,x_2)=0,
$$
because $c=\mu/(\lambda+\mu)$. Thus $\widetilde u(x_1,x_2)$ satisfies the Lam\'e system in $\Omega$. This completes the proof.
\end{proof}
\subsection{Reflection principle for Navier equation}
In this subsection, we derive the reflection principle for the Navier equation, which is based on the  corresponding reflection principle for the Lamé system established in the previous subsection.

   The construction follows from the fundamental solution and the corresponding Green's tensor.
   We recall that the fundamental solution $\Phi (x,y)$ for the Lamé operator \eqref{Lame oeprator} in $\mathbb{R}^2$ is given by (see {\cite[chapter-2.2]{hsiao2021boundary}})
\begin{equation}
\label{eq:fundamental}
\Phi(x,y)
=
\frac{1}{4\pi}
\left\{
-\frac{3\mu+\lambda}{\mu(\lambda+2\mu)} \ln |x-y| \, I_2
+
\frac{\lambda+\mu}{\mu(\lambda+2\mu)}
\frac{(x-y)\otimes(x-y)}{|x-y|^2}
\right\},
\  \   x \neq y,
\end{equation}
where $I_2$ denotes the identity matrix of order $2$ and the $ij$th entries of the matrix obtained after taking the tensor product of vectors $x=(x_1,x_2), y=(y_1,y_2)$ in $\mathbb{R}^2$ are given by  
$
(x \otimes y)_{ij} := x_i y_j\ \text{for} \  i,j = 1,2.
$
 Let $G(x,y)$ be the Green's tensor for the Lam\'e system \eqref{Lame oeprator}
in the half-plane $\mathbb{R}^2_{+} := \{x : x_1>0 \}$, which satisfies the traction-free boundary condition
\eqref{eq:traction_recall} on $\Lambda:=\{x_1=0\}$. Based on a result deduced in {{\cite[Corollary 2.2]{elschner2019uniqueness}}}, the Green's tensor $G(x,y)$  to the Lam\'e system in the half-plane $\mathbb{R}^2_{+}$ with traction-free boundary condition  can be expressed in terms of $\Phi(x,y)$ as
\begin{equation}
\label{eq:green}
G(x,y)
=
\Phi(x,y) + \mathcal{B}_x \Phi(R_\Lambda x,y),
\  \   x,y \in \mathbb{R}^2_{+}, \  x \neq y,
\end{equation}
where $R_
\Lambda$ denotes the reflection operator across $\Lambda$, and $\mathcal{B}_x$ is the extension operator acting with respect to the variable $x$ and it is same as given in \eqref{eqn:Extension}.

\medskip

 We now state the reflection principle for the Navier equation as follows.
\begin{theorem}
\label{thm:reflection nav}
Let $\gamma \subset \Lambda$ and $\Omega =\Omega^+ \cup \gamma \cup \Omega^-$ be a symmetric domain as defined before. Suppose that $u(x_1,x_2)=(u_1(x_1,x_2),u_2(x_1,x_2))$ is a solution to the following Navier equation satisfying the traction-free boundary condition 
\begin{equation}
\label{eq:3.1}
\begin{cases}
\mathcal{N}_\omega u := \mathcal{L}u+\omega^2 u=0\ \text{in }\Omega^+ ,\\
 T_{\nu}(u) := 2\mu\, \partial_\nu u 
+ \lambda\,(\nabla \cdot u)\nu
+ \mu\,(\nabla \times u)\, \nu^{\perp} = 0\ \text{on } \gamma.
\end{cases}
\end{equation}
Then $u$ can be analytically extended  to $\Omega^-$ as a solution to \eqref{eq:3.1}. Moreover, there exists an extension operator $\mathcal{\widetilde B}$  such that
$$
u(-x_1, x_2):= ( \mathcal{\widetilde B} u)(x_1, x_2), \  (x_1,x_2) \in \Omega^+, 
$$
where we refer to \eqref{eq:3.8} for a  definition of the operator $\mathcal{\widetilde B}$. 
\end{theorem}

\begin{proof}
Let $u$ be a solution to \eqref{eq:3.1}. Define a new function
\begin{equation}
\label{eq:v_def}
\widehat{u}(x)
:=
u(x)
-
\omega^{2}
\int_{\Omega^{+}} G(x,y)\,u(y)\,dy,
\  \   x \in \Omega^{+},
\end{equation}
where $G(x,y)$ is the Green's tensor given by \eqref{eq:green}.
Now it is easy to check that $\widehat{u}$ satisfies the Lam\'e system with traction-free boundary condition, that is,
\begin{equation}
\label{eq:v_bvp}
\begin{cases}
\mathcal{L}\widehat{u}:=\mu \Delta \widehat{u} + (\lambda+\mu)\nabla(\nabla\cdot \widehat{u}) = 0
& \text{in }\, \Omega^{+}, \\
T_{\nu}(\widehat{u}):=2\mu\,\partial_\nu \widehat{u} + \lambda (\nabla\cdot \widehat{u})\,\nu
+ \mu (\nabla \times \widehat{u})\,\nu^{\perp} = 0
&  \text{on}\, \gamma.
\end{cases}
\end{equation}
Applying Theorem ~\ref{thm:reflection} and the definition of the operator $\mathcal{B}$ in \eqref{eqn:Extension}, we conclude that $\widehat{u}$ can be extended analytically to $\Omega^-$ as a solution to the Lam\'e system. After denoting the extended function again by $\widehat{u}$,
we obtain
\begin{equation}
\label{eq:v_extension}
\widehat{u}(-x_1,x_2) = (\mathcal{B}\widehat{u})(x_1,x_2),
\  x \in \Omega.
\end{equation}
Now substitute \eqref{eq:v_def} in  \eqref{eq:v_extension},
to deduce that  
\begin{align}
\label{eq:3.8}
u(-x_1,x_2)
&=
(\mathcal Bu)(x)
-
\omega^{2}
\int_{\Omega^{+}} \mathcal{B}_x G(x,y)\,u(y)\,dy
\nonumber \\
&\ 
+
\omega^{2}
\int_{\Omega^{+}} G((-x_1,x_2),y)\,u(y)\,dy
=: (\widetilde{\mathcal{B}}u)(x),
\  \   x \in \Omega^{+}.
\end{align}
The representation \eqref{eq:3.8} defines an extension of $u$ from $\Omega^+$ to $\Omega^-$. By construction, this extension satisfies the Navier equation in $\Omega^-$, and hence in the full domain $\Omega$. This finishes the proof of Theorem~\ref{thm:reflection nav}.
\end{proof}

\section{Application of Reflection principle}\label{sec:application}

\noindent In order to apply the reflection formula \eqref{eq:3.8} to the inverse obstacle scattering problem considered in this manuscript, we need to analytically extend to the whole space $\mathbb{R}^2$ the solution of the traction-free boundary value problem for the Navier equation in a sector. To describe it, we use polar coordinates \((r,\theta)\) in \(\mathbb{R}^2\) and for a given $\theta_0 \in (0,\tfrac{\pi}{2}]$, we define
\begin{equation}\label{eq:Aj_def}
 A_j := \left\{ (r,\theta)  :  r>0,\ \theta \in (-\tfrac{\pi}{2}+j\theta_0, -\tfrac{\pi}{2}+(j+1)\theta_0) \right\}
,
\end{equation}
and
\begin{equation}\label{eq:Lj_def}
\Lambda_j := \left\{ (r,\theta)  :  r>0,\ \theta = -\tfrac{\pi}{2}+ j\theta_0 \right\}, 
\  \   j\in \mathbb{N}\cup\{0\}. 
\end{equation}
Clearly, we have
\begin{equation*}\label{eq:boundary_Aj}
\partial A_j = \Lambda_j \cup \Lambda_{j+1}.
\end{equation*}
We assume that the normal direction $\nu$ at $\Lambda_j$ is directed into the sector $A_j$.
The geometry of the sector $A_0$ and of the domain $D$ is illustrated in Figure~\ref{fig:sector}.

\begin{figure}[H]
\centering
\begin{tikzpicture}[scale=1.0]
  \useasboundingbox (-2.0,-3.1) rectangle (2.4,2.4);
  \fill[gray!10] (0,0) -- (240:3) arc (240:420:3) -- cycle;
  \draw[dashed,thick] (0,0) -- (240:3);
  \draw[dashed,thick] (0,0) -- (420:3);
  \fill[blue!18] (0,0) -- (270:2.2) arc (270:330:2.2) -- cycle;
  \fill[orange!22] (0,0) -- (330:2.2) arc (330:390:2.2) -- cycle;
  \draw[thick] (0,0) -- (270:2.2);
  \draw[thick] (0,0) -- (330:2.2);
  \draw[thick] (0,0) -- (390:2.2);
  \fill (0,0) circle (1.6pt);
  \node[below left,xshift=1pt,yshift=8pt] at (0,0) {$P_0$};
  \node at (300:1.35) {$A_0$};
  \node at (0:1.5) {$A_1$};
  \node[left] at (270:2.55) {$\Lambda_0$};
  \node[below right] at (330:2.45) {$\Lambda_1$};
  \node[above right] at (390:2.45) {$\Lambda_2$};
  \node[left] at (252:2.6) {$D$};
\end{tikzpicture}
\caption{The domain $D$ (lightly shaded, with dashed boundary) contains the sector $A_0$ in its interior. The sector $A_1$ is the reflection of $A_0$ across the line $\Lambda_1$. The domain $D$ is unbounded and extends beyond the edges of the picture.}
\label{fig:sector}
\end{figure}
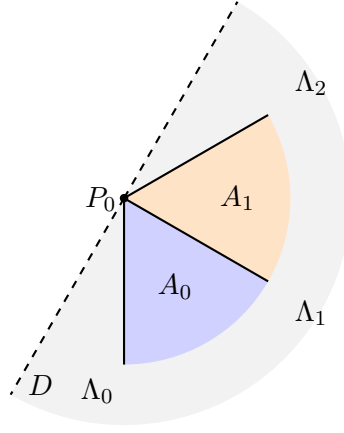

\begin{lemma} \label{lem:2}
Let \(D \subset \mathbb{R}^2\) be an unbounded domain that contains the sector \(A_0\). Assume that \(u\) satisfies the Navier equation \eqref{Equation for Incident field} in $D$, with the traction-free boundary condition \eqref{eq:8} on $\Lambda_0 \cup \Lambda_1$, that is,
\begin{equation}
\label{eq:aa}
\begin{cases}
\mathcal{N}_\omega(u) = 0\ \text{in } D, \\
T_\nu(u) = 0\ \text{on } \Lambda_0 \cup \Lambda_1.
\end{cases}
\end{equation}
Then $u$ admits an analytic extension to the whole space $\mathbb{R}^2$.

\begin{proof}
Let $u$ satisfy the boundary value problem \eqref{eq:aa}. Since $A_0 \subset D$, we may consider the restriction of $u$ to $A_0$.
By Theorem~\ref{thm:reflection nav}, $u$ admits an analytic extension across the boundary $\Lambda_1$ into the sector $A_1$, given by
\begin{equation*}
u(x) = (\widetilde{\mathcal{B}}_1 u)\big(R_{\Lambda_1} x\big), \  x \in A_1,
\end{equation*}
where $\widetilde{\mathcal{B}}_1$ is the extension operator associated with the traction-free boundary condition $T_\nu(u)=0$ on $\Lambda_1$. 
Introduce the function
\begin{equation*}
v(x):=R_{\Lambda_1}u(R_{\Lambda_1}x) =R_{\Lambda_1} (\widetilde{\mathcal{B}}_1 u)(x), \  x \in A_0 \cup \Lambda_1 \cup A_1.
\end{equation*}

Since the reflection $R_{\Lambda_1}$ maps $\Lambda_0$ onto $\Lambda_2$, and since both the Navier equation and the traction operator are invariant under reflection, it follows from Proposition \ref{pro} that the field $v$ obtained above satisfies the traction-free boundary condition on $\Lambda_2$.

\noindent Applying Theorem~\ref{thm:reflection nav} again, now across the side $\Lambda_2$, we obtain an analytic extension of $u$ from $A_1$ into $A_2$, and this extension satisfies the traction-free boundary condition on $\Lambda_3$. Repeating this argument inductively, we obtain an analytic extension of $u$ from $A_0$ into the successive sectors $A_j$ for $j\ge1$, with $T_\nu(u)=0$ on each of the rays $\Lambda_j$.%

Now, since $\theta_0 \in (0,\tfrac{\pi}{2}]$, there exists an integer $k \geq 2$ such that
$
k\theta_0 \geq \pi.
$
By iterating the above reflection-extension procedure $k-1$ times, $u$ can be analytically extended to the right half-plane $\{x_1 \geq 0\}$.
Moreover, the extended function satisfies the traction-free boundary condition
$$
T_\nu(u) = 0 \  \text{on } \{x_1 = 0\}.
$$
Applying Theorem \ref{thm:reflection nav} again across the line $\{x_1=0\}$, we further extend $u$ from $\mathbb{R}^2_+$ to the entire space $\mathbb{R}^2$. This completes the proof of the lemma.
\end{proof}
\end{lemma}
 We now turn to the case where the boundary is locally flat. In this setting, the extension argument relies on reflection across a straight line rather than corner-based analysis.
The following lemma provides the corresponding extension result for solutions to the Navier equation under traction-free boundary conditions.

To state the result precisely, we fix the following notation. Let $D \subset \mathbb{R}^2$ be a bounded polygonal domain and let $G \subset \mathbb{R}^2 \setminus \overline{D}$ denote an unbounded exterior domain. Consider a line segment $\gamma \subset \partial D \cap \partial G$, and let $\Lambda$ be the straight line containing $\gamma$. We denote by $R_{\Lambda}(D)$ the reflection of $D$ across $\Lambda$, and by $\widetilde B$ the extension operator associated with the traction-free boundary condition on $\gamma$ introduced in Theorem~\ref{thm:reflection nav}.
\begin{lemma} \label{lem:3}
Assume that \(u\) satisfies the Navier equation \eqref{Equation for Incident field} in $D \cup \gamma \cup G$, with the traction-free boundary condition \eqref{eq:8} on $\partial D$, that is,
\begin{equation}
\label{eq:abc}
\begin{cases}
\mathcal{N}_\omega(u) = 0 \  \text{in } D \cup \gamma \cup G, \\
T_\nu(u) = 0 \  \text{on } \partial D.
\end{cases}
\end{equation}
Then the following statement holds
\begin{enumerate}
    \item[(i)] The reflected field $w :=R_{\Lambda}\circ\mathcal{\widetilde B}u$ admits an analytic extension to $G \cup R_\Lambda(D)$ as a solution to the Navier equation with traction-free boundary condition  \eqref{eq:8}  on $\partial R_{\Lambda}(D)$. 
    \item[(ii)] Suppose that there exists a half-plane contained in $G$ whose boundary line extends one side of $\partial R_{\Lambda}(D)$. Then both $w$ and $u$ admit analytic continuation to the entire plane $\mathbb{R}^2$.
\end{enumerate}
\end{lemma}
\begin{proof}
(i) Since $u$ satisfies the Navier equation in $D \cup \gamma \cup G$ and the traction-free boundary condition on $\partial D$, in particular we have
$$
T_\nu(u) = 0 \  \text{on } \gamma.
$$
Let $\Lambda$ be the straight line containing $\gamma$. Then, in a neighbourhood of $\gamma$, the domain $D \cup \gamma \cup G$ is locally symmetric with respect to $\Lambda$.
By the reflection principle for the Navier equation (Theorem~\ref{thm:reflection nav}), there exists an extension operator $\widetilde{B}$ such that the reflected field
$
\widetilde{B}u(R_\Lambda x)
$ 
defines an analytic continuation of $u$ across $\Lambda$, implying that $u(x)=\widetilde{B}u(R_\Lambda x)$. Hence, $w =R_{\Lambda}\circ\mathcal{\widetilde B}u=R_{\Lambda}\circ u(R_{\Lambda}\circ x)$ is well-defined in $G \cup R_\Lambda(D)$.
Applying Proposition \ref{pro}, one obtains that the extended field $w$ satisfies
$$
\mathcal{N}_\omega(w) = 0 \  \text{in } G \cup R_\Lambda(D),
$$
and inherits the traction-free boundary condition on the reflected boundary, that is,
$$
T_\nu(w) = 0 \  \text{on } \partial R_\Lambda(D).
$$
This proves part (i).

(ii) Assume now that there exists a half-plane $H \subset G$ whose boundary line coincides with an extension of one side of $\partial R_\Lambda(D)$.

From part (i), the function $w$ is already defined in $G \cup R_\Lambda(D)$ and satisfies the Navier equation together with the traction-free condition on $\partial R_\Lambda(D)$. In particular, along the straight line forming the boundary of the half-plane $H$, the traction-free condition holds.

\noindent We may therefore apply the reflection principle (Theorem~\ref{thm:reflection nav}) once again across this line. This yields an analytic continuation of $w$ from the half-plane $H$ to its reflected counterpart. Since $H$ is a half-plane, this reflection extends the solution across an entire straight line. Repeating this argument if necessary, we obtain an analytic continuation of $w$ to the whole space $\mathbb{R}^2$. Finally, since $w=R_{\Lambda}\circ u\circ R_{\Lambda}$, the same continuation applies to $u$. Hence both $u$ and $w$ admit analytic continuation to the entire plane $\mathbb{R}^2$.
This completes the proof.
\end{proof}

\begin{proof}[\bfseries Proof of Theorem~\ref{thm:main_result}]
 The proof proceeds by contradiction and relies on the reflection principle for the Navier equation together with the unique continuation property of elastic fields. 
 
\noindent Let $D_1, D_2 \subset \mathbb{R}^2$ be two bounded, connected polygonal obstacles. For each $j=1,2$, we denote by $u_j^\infty$, $u_j^{\mathrm{sc}}$, and $u_j$ the elastic far-field pattern, the scattered field, and the total field corresponding to $D_j$, respectively, generated by the fixed incident plane wave given by \eqref{eq:plane}.

\noindent Assume that $D_1 \neq D_2$, while the corresponding far--field patterns coincide, that is,
\begin{equation}
\label{eq:equal_farfield}
u_1^\infty(\widehat x) = u_2^\infty(\widehat x)\ \text{for all} \       \widehat x\in\mathbb{S}^1.
\end{equation}
Then, by Rellich's lemma (e.g., Lemma 5.8 in \cite{hahner1998scattering}),  the corresponding total fields agree in the unbounded component $E$ of $\mathbb{R}^2 \setminus \overline{D_1 \cup D_2}$, that is,
\begin{equation}
\label{eq:3.13}
u_1 = u_2
\  \text{ in } E
\end{equation}
Since $D_1 \neq D_2$, the sets $\partial D_1$ and $\partial D_2$ do not coincide. Hence, we can assume that $$ \partial D_2 \cap \partial E \cap D_1^{\,c} \neq \emptyset \  \text{or} \  \partial D_1 \cap \partial E \cap D_2^{\,c} \neq \emptyset. $$ Without loss of generality, we assume that $$ \partial D_2 \cap \partial E \cap D_1^{\,c} \neq \emptyset. $$
Thus, using the fact that \(\partial D_2\) consists of finitely many straight line segments together with their endpoints (corner points), we can choose a point 
$P_0 \in \partial D_2 \cap \partial E \cap D_1^{\,c}$ which is either a corner point of \(\partial D_2\) or lies in the interior of one of the edges of $\partial D_2$.
Hence, we consider the two different cases for the aforementioned possibilities of the point $P_0$.  
\begin{case}\label{case: corner point} 
$P_0$ is a corner point of $\partial D_2$. 
\end{case}
If  $P_0 \in \partial D_2 \cap \partial E \cap D_1^{\,c}$ is a corner point of $D_2$, then the  two adjacent edges of $\partial D_2$ meet at $P_0$. Extending these edges beyond $P_0$ into $D_1^{\,c}$, we obtain two half-lines, denoted by $\Lambda_0$ and $\Lambda_1$; see Figure~\ref{fig:case1}.
Since \(u_1 = u_2\) in \(E\) and  \(u_2\) satisfies the traction-free boundary condition on \(\partial D_2\), it follows that \(u_1\) also satisfies the following 
\begin{equation}
\label{eq:traction_case}
T_\nu(u_1)=2\mu\,\partial_\nu u_1
+ \lambda (\nabla \cdot u_1)\,\nu
+ \mu (\nabla \times u_1)\,\nu^{\perp}
= 0
\  \text{on } \Lambda_0 \cup \Lambda_1.
\end{equation}
 The choice of the normal direction on $\Lambda_0$ and $\Lambda_1$ is arbitrary,
because reversing the normal direction does not change the traction-free
boundary condition.

\begin{figure}[H]
\centering
\includegraphics[width=0.37\textwidth]{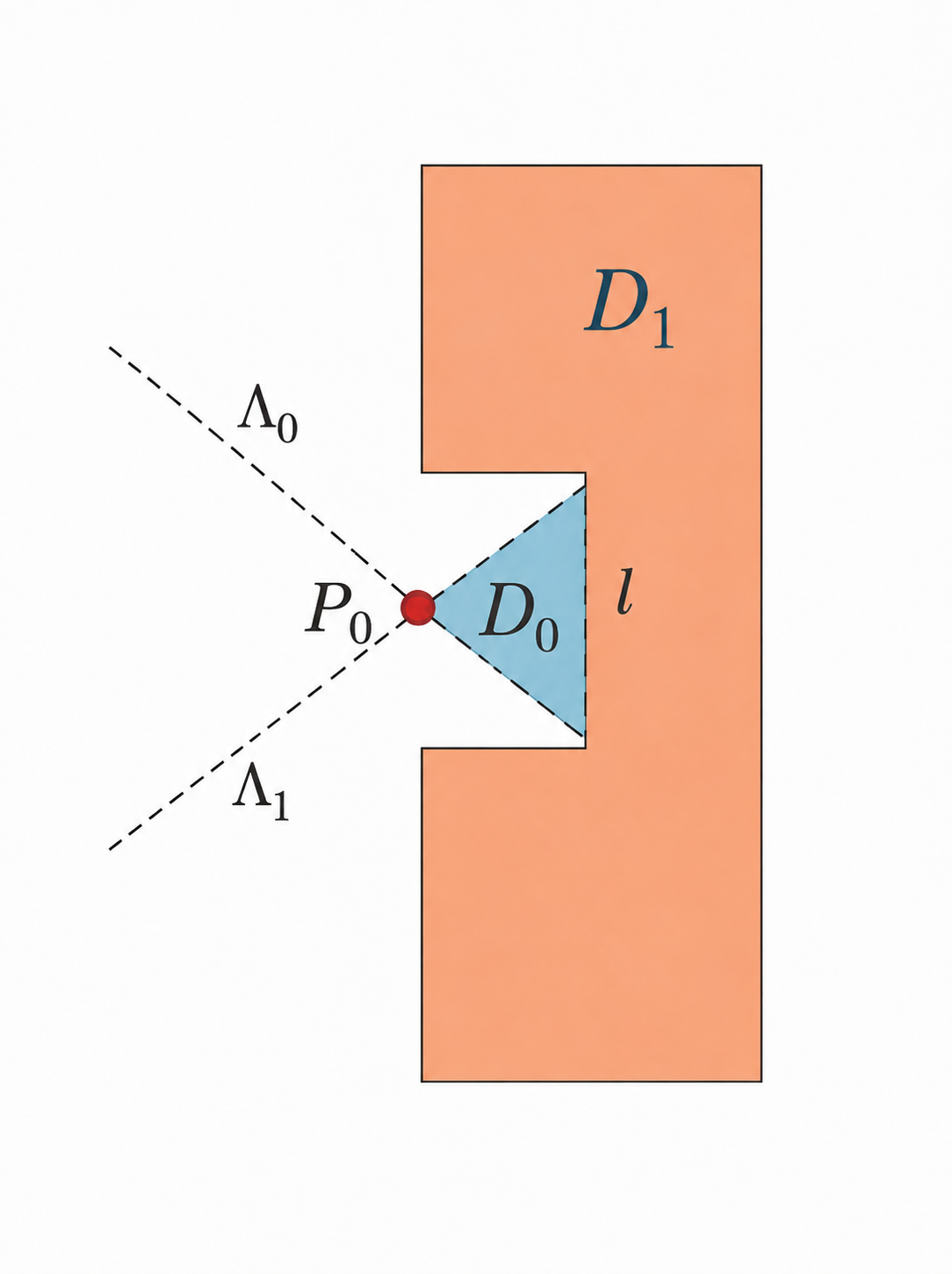}
\caption{Schematic of the polygonal obstacle $D_1$ and $D_2=D_1 \cup l \cup D_0$, where
$D_0$ represent the region separating $D_1$ from $D_2$. Near the corner point $P_0$, two edges of $D_2$ can be extended into half-lines lying in the exterior domain $D_1^{\,c}$.}
\label{fig:case1}
\end{figure}
\noindent By the translation and rotational invariance of the isotropic Navier equation, we may, without loss of generality, translate and rotate the coordinate system so that $P_0$ is located at the origin and that the two half-lines $\Lambda_0$ and $\Lambda_1$ takes the form given in \eqref{eq:Lj_def}.  
Let $A_0$ denote the infinite sector bounded by these two half-lines and lying in the exterior of $D_2$.
By the connectedness of $D_1$, the infinite sector $A_0$ must be completely contained in $E \cap D_1^{\,c}$, which implies that the total field
$$
u_1 = u^{\mathrm{in}} + u_1^{\mathrm{sc}}
$$
satisfies the Navier equation \eqref{Navier operator} in a neighboring area of the sector $A_0$.
Applying Lemma~\ref{lem:2}, we can extend $u_1$ to an entire solution in $\mathbb{R}^2$. It follows that the scattered field $u_1^{\mathrm{sc}}$ is an entire radiating solution. Hence, by the radiation condition and Rellich's lemma, $u_1^{sc} \equiv 0$. 
Consequently, the total field is
$
u_1(x) = u^{\mathrm{in}}(x) = c_p\, d\, e^{\mathrm{i}k_p x\cdot d}
+ c_s\, d^\perp\, e^{\mathrm{i}k_s x\cdot d} \ \text{in }D_1^c
$. Since $u_1$ satisfies the traction-free boundary condition on $\partial D_1$, it follows that

$$
T_\nu(u^{\mathrm{in}}(x)) = 0, \  \   x \in \partial D_1.
$$
Writing $u_p^{in} := de^{\mathrm{i}k_p x\cdot d}$ and $u_s^{in} := d^\perp\, e^{\mathrm{i}k_s x\cdot d}$, the linearity of the traction operator gives
$$
 c_pT_\nu(u_p^{\mathrm{in}}(x)) + c_sT_\nu(u_s^{\mathrm{in}}(x))=0, \  x\in \partial D_1.
$$
For the P-wave, since d is a constant unit vector, we start with computing
$$\nabla u_p^{\mathrm{in}}(x) = i k_p (d \otimes d) e^{i k_p x \cdot d},\  \  \
\partial_\nu u_p^{\mathrm{in}}(x) = i k_p  (d \cdot \nu)\, d\, e^{i k_p x \cdot d},$$
and $$
\nabla \cdot u_p^{\mathrm{in}}(x) = i k_p e^{i k_p x \cdot d},
\  \  \
\nabla \times u_p^{\mathrm{in}}(x) = 0.
$$
Hence,
\begin{equation}\label{eq:Ptraction}
T_\nu(u_p^{\mathrm{in}}(x)) =
2\mu i k_p  (d \cdot \nu)\, d\, e^{i k_p x \cdot d}
+ \lambda i k_p  \nu\, e^{i k_p x \cdot d}
\end{equation}
In a similar way, for the S-wave, 
we compute
$$
\nabla u_s^{\mathrm{in}}(x) = i k_s  (d \otimes d^\perp) e^{i k_s x \cdot d},
\  \  \  \partial_\nu u_s^{\mathrm{in}}(x) = i k_s (d \cdot \nu)\, d^\perp\, e^{i k_s x \cdot d},$$
and
$$
\nabla \cdot u_s^{\mathrm{in}}(x) = 0,
\  \  \ 
\nabla \times u_s^{\mathrm{in}}(x) = i k_s  e^{i k_s x \cdot d}.
$$
Therefore,
\begin{equation}\label{eq:Stracion}
T_\nu(u_s^{\mathrm{in}}(x)) =
2\mu i k_s  (d \cdot \nu)\, d^\perp\, e^{i k_s x \cdot d}
+ \mu i k_s  \nu^\perp\, e^{i k_s x \cdot d}.
\end{equation}
Combining \eqref{eq:Ptraction} and \eqref{eq:Stracion}, the traction-free
condition on $\partial D_1$ yields
\begin{align*}
T(u^{in}(x))
&=
2\mu i k_p c_p (d\cdot \nu)d\, e^{i k_p x\cdot d}
+ \lambda i k_p c_p \nu\, e^{i k_p x\cdot d} \\
&\ 
+ 2\mu i k_s c_s (d\cdot \nu)d^\perp\, e^{i k_s x\cdot d}
+ \mu i k_s c_s \nu^\perp\, e^{i k_s x\cdot d}= 0\ \mbox{on}\ \partial D_1, 
\end{align*}
which after combining various terms implies 
$$ 
i k_p c_p e^{i k_p x\cdot d}
\Big( 2\mu(d\cdot \nu)d + \lambda\nu \Big) + i k_s c_s e^{i k_s x\cdot d}
\Big( 2\mu(d\cdot \nu)d^\perp + \mu\nu^\perp \Big)=0\ \text{on} \ \ \partial D_1.$$
Set
$$
A_\nu:=2\mu(d\cdot\nu)d+\lambda\nu,
\ \ \ \
B_\nu:=2\mu(d\cdot\nu)d^\perp+\mu\nu^\perp .
$$
Then
$$
ik_p c_p e^{ik_p x\cdot d}A_\nu
+
ik_s c_s e^{ik_s x\cdot d}B_\nu=0 \ \ \text{on} \ \ \partial D_1.
$$
We now exploit the polygonal structure of $D_1$. Since $D_1$ is a bounded polygonal obstacle, its boundary $\partial D_1$ consists of finitely
many open line segments (edges). These edges cannot all be parallel to
the same direction. Therefore, there exists an edge
$\Gamma\subset\partial D_1$
whose unit tangent vector $\tau$ satisfies
$$
\tau\cdot d\neq0.
$$
Since $\Gamma$ is a straight line segment, both the unit tangent vector
$\tau$ and the outward unit normal vector $\nu$ are constant on $\Gamma$.
Let $L>0$ denote the length of $\Gamma$. We parameterize $\Gamma$ by $$
\Gamma=\{x_0+t\tau:\;0<t<L\},
$$ where $x_0 \in \Gamma$ is fixed. For any $x\in \Gamma$, we can write $x= x_0 + t\tau$ for some $t\in (0,L)$. Hence,
$$
x\cdot d=x_0\cdot d+t(\tau\cdot d).
$$
It follows that
$$
e^{ik_px\cdot d}
=
e^{ik_px_0\cdot d}e^{ik_p(\tau\cdot d)t}, 
\ \ and\ \
e^{ik_sx\cdot d}
=
e^{ik_sx_0\cdot d}e^{ik_s(\tau\cdot d)t} \ \  \text{for} \ \ x\in \Gamma
$$
Since $k_p\neq k_s$ and $\tau\cdot d\neq0$, these two functions are linearly
independent on the interval $(0,L)$, and consequently on $\Gamma$. Therefore,
$$
c_p A_\nu=0, \ \ \text{and}
\quad
c_s B_\nu=0
\ \ \text{on }\Gamma .$$
From here, we have the following possibilities. 
Either $c_p=0$ or $A_\nu=0$ on $\Gamma$ and either $c_s=0$ or $B_\nu=0$ on $\Gamma$. We start with the second possibility i.e.  either $c_s=0$ or $B_\nu=0$ on $\Gamma$.  Now if $c_s\neq 0$ then
$ B_\nu=0\ \text{on }\Gamma$ and hence $$
2\mu(d\cdot\nu)d^\perp+\mu\nu^\perp=0\ \text{on }\Gamma.
$$
Dividing by $\mu>0$, we obtain
$
2(d\cdot\nu)d^\perp+\nu^\perp=0\ \text{on }\Gamma
$.
Taking the dot product with $d^\perp$, we get
$$
2(d\cdot\nu)+\nu^\perp\cdot d^\perp=0\ \text{on }\Gamma. 
$$
Since
$
\nu^\perp\cdot d^\perp= \nu\cdot d\ \text{on }\Gamma
$
we obtain
$
(d\cdot\nu)=0\ \text{on }\Gamma
$.
Substituting this into
$$
2(d\cdot\nu)d^\perp+\nu^\perp=0\ \text{on }\Gamma,
$$
gives
$
\nu^\perp=0\ \text{on }\Gamma.
$
This is impossible because  $\nu^\perp$ is a unit vector and hence 
$
c_s=0
$. Now  using this in the traction free identity, we arrive at 
$$
ik_p c_p e^{ik_p x\cdot d}A_\nu=0\ \text{on }\partial D_1.
$$
Since $ik_p e^{ik_p x\cdot d}\neq0$, we have
$$
c_p A_\nu=0\ \text{on} \ \partial D_1.
$$
Now if  $c_p\neq0$, then 
$
A_\nu=0\ \text{on } \partial D_1,
$
that is,
$$
2\mu(d\cdot\nu)d+\lambda\nu=0\ \text{on }\partial D_1.
$$
Taking the dot product with $d$, we get that 
$
(2\mu+\lambda)(d\cdot\nu)=0\ \text{on }\partial D_1.
$
Since $\lambda+2\mu>0$, it follows that
$$
d\cdot\nu=0\ \text{on }\partial D_1.
$$
Thus, the outward normal to every edge of $D_1$ is perpendicular to $d$, which means that all edges of $D_1$ are parallel to $d$. This is impossible for the boundary of a bounded polygonal obstacle, because  a bounded polygon must have at least two non-parallel edge directions. Hence
$
c_p=0.
$
Combining this with $c_s=0$, we obtain
$$
c_p=c_s=0,
$$
which contradicts the assumption that 
$$
|c_p|+|c_s|\neq0.
$$
This contradiction proves that $D_1=D_2$.
\begin{case}\label{case: interior point}
    When $P_0$ is an interior point of $\partial D_2$.
\end{case}
Since, in this case, $P_0$ is not a corner of $D_2$, there exists a sufficiently small neighborhood $U$ of $P_0$ in $\mathbb{R}^2$ such that   $U\cap \partial D_2$ contains a  straight line segment. Consequently, we can choose a finite line segment $\gamma_0 \subset \partial D_2 \cap \partial E \cap D_1^c$ containing $P_0$, whose two endpoints lie on $\partial D_1 \cap \partial E$; see Figure~\ref{fig:case2}.

\begin{figure}[htbp]
    \centering
    \begin{minipage}{0.30\textwidth}
        \centering
        \includegraphics[width=\linewidth]{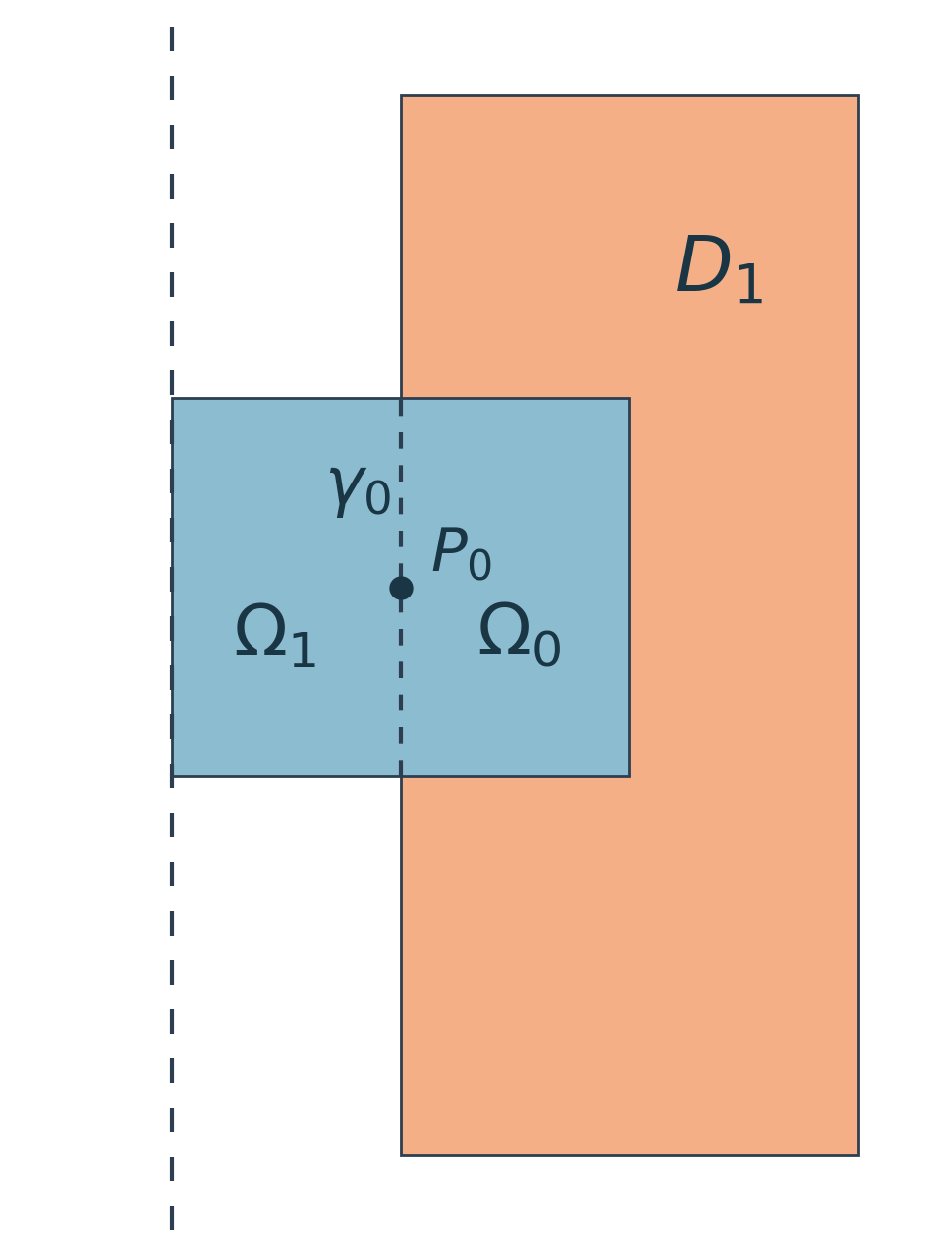}
    \end{minipage}
    \hspace{0.15\textwidth}
    \begin{minipage}{0.37\textwidth}
        \centering
        \includegraphics[width=\linewidth]{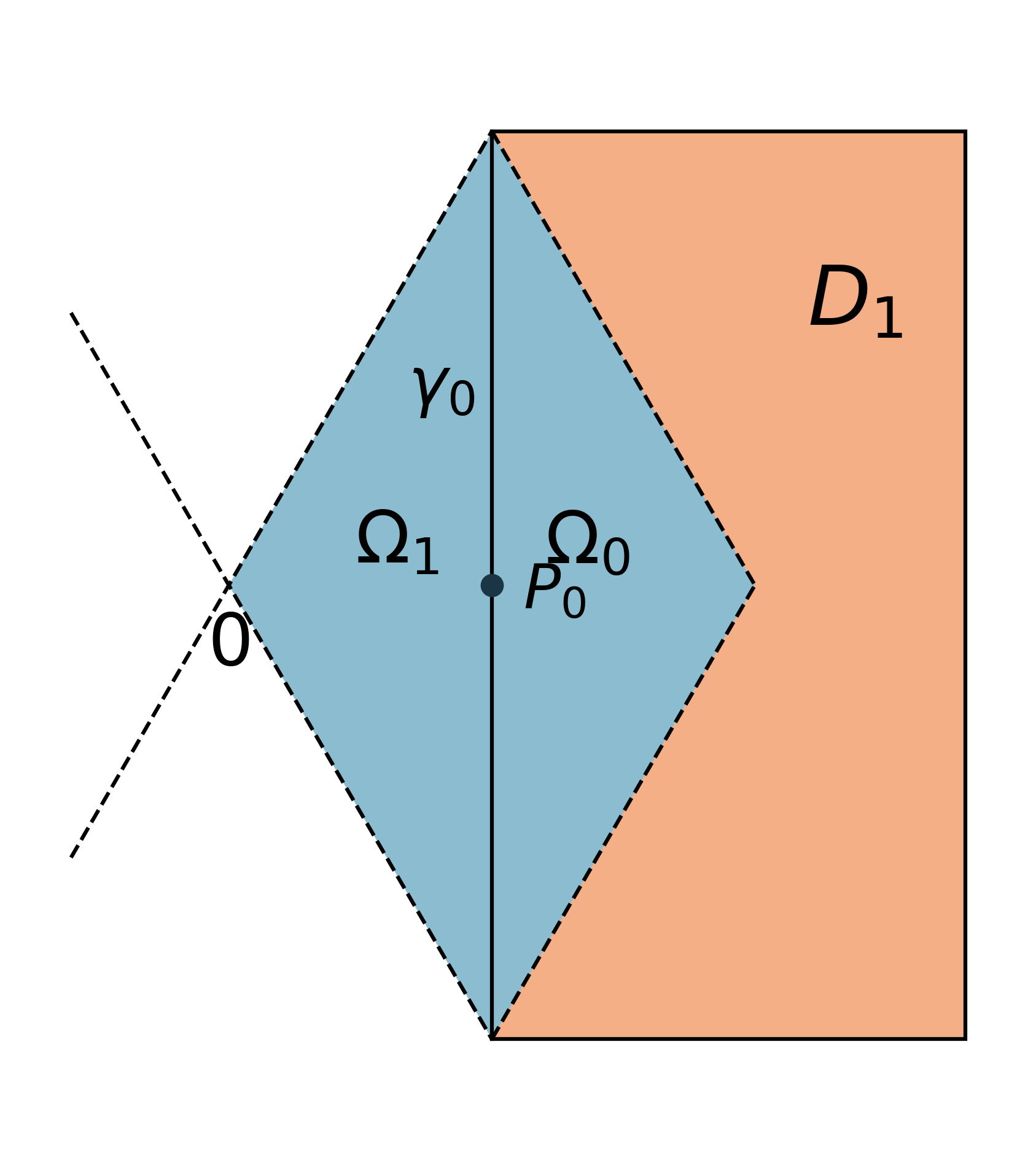}
    \end{minipage}
    
    \caption{Schematic of the polygonal obstacle $D_1$ and $D_2=D_1 \cup (\overline{\Omega_0}\setminus \gamma_0)$, where
    $\Omega_0$ denotes the region separating $D_1$ from $D_2$. The domain $\Omega_1$ is the reflection of $\Omega_0$ across $\gamma_0$. Left: one edge of $\partial \Omega_1$ extends to a straight line contained in $D_1^{c}.$ Right: two adjacent edges of $\partial \Omega_2$ extend to half-lines contained in $D_1^{c}$ starting from the corner $O$, forming an infinite sector in $D_1^{c}$. }
    \label{fig:case2}
\end{figure}

On this segment, the total field $u_1$ satisfies the traction-free boundary condition $$
T_\nu(u_1)= 0\ \text{on } \gamma_0 .
$$ 
 Let $\Omega_0$ and $G_0$ denote, respectively, the bounded and unbounded connected components of $D_1^{\,c} \setminus \gamma_0$. By construction, $\partial\Omega_0$ consists of $\gamma_0$ together with a portion of $\partial D_1$.
Consequently,
\begin{equation*}
\begin{cases}
\mathcal{N}_\omega(u_1) = 0\ \text{in }\Omega_0 , \\
T_\nu(u_1)= 0\ \text{on } \partial \Omega_0 .
\end{cases}
\end{equation*}
To obtain a contradiction, we prove by following the approach used in \cite{elschner2019uniqueness}, that  the total field $u_1$ admits an analytic continuation to the entire plane $\mathbb{R}^2$ and finally utilizing the argument used in Case \ref{case: corner point}, we conclude the proof of desired uniqueness result. 

 Choose a continuous injective curve in the unbounded component $G_0$ and starting from $P_0\in \gamma_0$. In particular, we assume  $$Z:[0,\infty)\longrightarrow G_0$$ be a continuous curve such that $Z(0)=P_0 $ and $|Z(t)|\to\infty$ as $t\to\infty$.  The curve  $Z$ will be used to determine successively the sides across which the analytic continuation is performed.
 
 Let $\Lambda_0$ denote the straight line containing $\gamma_0$, and let $R_{\Lambda_0}$ denote the reflection with respect to  $\Lambda_0$. We denote by $\mathcal{\widetilde B}_0$, the extension operator associated with the traction--free boundary condition on $\gamma_0 \subset \Lambda_0$. By the rotational and translation invariance of the Navier equation, $\mathcal{\widetilde B}_0$ is obtained from the extension operator in \eqref{eq:3.8} after a suitable rigid transformation.

 Set $
u_1^{(0)}(x) := u_1(x)
$ for $x\in \Omega_0$. By Lemma \ref{lem:3}-(i), the function
$
u_1^{(0)}
$ can be analytically continued across $\Lambda_0$ into the reflected domain $\Omega_1 :=R_{\Lambda_0}(\Omega_0)$ through the extension operator $\mathcal{\widetilde B}_0$. We define the extended field $u_1^{(1)}(x) :=(R_{\Lambda_0}\circ \mathcal{\widetilde B}_0)\; u_1^{(0)}(x)$ in $\Omega_1 \cup D_1^{\,c}$. Then $u_1^{(1)} = R_{\Lambda_0}\circ u_1^{(0)}\circ R_{\Lambda_0}$ in $\Omega_1$, and by Proposition \ref{pro}, $u_1^{(1)}$ satisfies the Navier equation in $\Omega_1$ together with the traction-free boundary condition on $\partial \Omega_1$, that is,
\begin{equation*}
\begin{cases}
\mathcal{N}_{\omega}(u_1^{(1)})= 0
\  \text{in } \Omega_1 , \\
T_\nu(u_1^{(1)})= 0
\  \text{on } \partial \Omega_1 .
\end{cases}
\end{equation*}
As $\Omega_1$ is bounded and the curve $Z(t)$ tends to infinity, we have
$
\partial \Omega_1 \cap \{ Z(t) : t > 0 \} \neq \emptyset .
$
We set $$
t_1 := \sup \{ t : \Omega_1 \cap Z(t) \neq \emptyset \} 
 ,\  \   P_1 := Z(t_1).$$ By the continuity of the curve, $P_1$ lies on $\partial \Omega_1$.  Moreover, since $Z(t)$ is injective, we have $P_1 \neq P_0$ and $\Omega_1 \cap \{ Z(t) : t > t_1 \} = \emptyset$.
Without loss of generality,  we may assume that 
$
P_1
$
is not a corner point of $\partial \Omega_1$ and hence we get that  $P_1$ lies on the open line segment $\gamma_1$  of $\partial \Omega_1$. 
Continuing like this, for  each $n \geq 2$, we denote by  $\Lambda_n$ straight line containing the open line segment $\gamma_n$ and let $R_{\Lambda_n}$ be the reflection with respect to  $\Lambda_n$. We further denote by $\mathcal{\widetilde B}_n$ the operator that extends solutions of the Navier equation across $\gamma_n \subset \Lambda_n$ under the traction-free boundary condition.

More precisely, for an open line segment  $\gamma_1 \subset \partial \Omega_1$ containing point  $P_1$, we define
$
\Omega_2 := R_{\Lambda_1}(\Omega_1) .
$
By Lemma \ref{lem:3} (i), the field $u_1^{(1)}$ admits an extension across $\gamma_1$ into $\Omega_2$. Consequently, we define the extended field
$
u_1^{(2)}(x) :=(R_{\Lambda_1}\circ \mathcal{\widetilde B}_1) u_1^{(1)}(x)$ for $x \in \Omega_2 \cup D_1^{c}$.
Then we have  
$$
u_1^{(2)}= (R_{\Lambda_1}\circ \mathcal{\widetilde B}_1)\circ (R_{\Lambda_0}\circ \mathcal{\widetilde B}_0) u_1\quad \mbox{in}\; D_1^{\,c} 
$$
 and
$
u_1^{(2)}(x) =R_{\Lambda_1}\circ u_1^{(1)}\circ R_{\Lambda_1}$ in $ \Omega_2 .
$

Hence, $u_1^{(2)}$ is a solution to the Navier equation in $\Omega_2$ together with the traction-free boundary condition on $\partial \Omega_2$, that is,
\begin{equation*}
\begin{cases}
\mathcal{N}_{\omega}(u_1^{(2)})= 0
\  \text{in } \Omega_2 , \\
T_\nu(u_1^{(2)})= 0
\  \text{on } \partial \Omega_2 .
\end{cases}
\end{equation*}
Continuing along the curve $Z(t)$, we choose a point
$
P_2 := Z(t_2)$ for some $ t_2 > t_1,$
 with $ P_2 \neq P_1$,
and a line segment $\gamma_2 \subset \partial \Omega_2$ such that
$
P_2 \in \gamma_2
\  \text{and} \ 
\Omega_2 \cap \{ Z(t) : t > t_2 \} = \emptyset .
$
Proceeding inductively, we can construct a sequence of polygonal domains
$$
\Omega_N := [R_{\Lambda_{N-1}} R_{\Lambda_{N-2}} \cdots R_{\Lambda_1}](\Omega_1), \  N \ge 1,
$$
and corresponding extension functions
$$
u_1^{(N)}(x) :=[(R_{\Lambda_{N-1}}\circ \mathcal{\widetilde B}_{N-1})\circ (R_{\Lambda_{N-2}}\circ \mathcal{\widetilde B}_{N-2})\cdots  (R_{\Lambda_1}\circ \mathcal{\widetilde B}_1)\circ (R_{\Lambda_0}\circ \mathcal{\widetilde B}_0)  u_1](x),
\  x \in D_1^{\,c} \cup \Omega_N ,
$$

such that $u_1^{(N)}$ is a solution to the Navier equation in $\Omega_N$ together with the traction-free boundary condition on $\partial \Omega_N$, that is,
\begin{equation*}
\begin{cases}
\mathcal{N}_{\omega}(u_1^{(N)})=0
\  \text{in } \Omega_N , \\
T_\nu(u_1^{(N)})= 0
\  \text{on } \partial \Omega_N .
\end{cases}
\end{equation*}
Continuing this construction, we assume that
$$\{ Z(t) : t > t_N \} \cap \Omega_N = \emptyset\ \text{for some} \  t_N > t_{N-1}$$ and let
$
P_N := Z(t_N) \in \gamma_N ,
$
where $\gamma_N \subset \partial \Omega_N$ is a line segment.
Since the curve $Z(t)$ remains in the connected set $D_1^c$ and tends to infinity, it follows from \cite[Lemma 3.3]{elschner2019uniqueness} that

$$
t_N \to \infty \  \text{as } N \to \infty .
$$
Consequently, after finitely many reflections, we get that  either the line $\Lambda_N$ lies entirely in $D_1^{\,c}$,
or the segment $\gamma_N$ together with one of its adjacent edges of $\partial \Omega_N$
extends to two half-lines in $D_1^{\,c}$. In the first situation, Lemma \ref{lem:3} (ii) implies that the field $u_1^{(N)}$ admits an analytic continuation to the entire plane $\mathbb{R}^2$. In the second situation, the two exterior half-lines bound an unbounded sector contained in $D_1^c$. Since $u_1^{(N)}$  satisfies the Navier equation in $D_1^c$ and traction--free boundary condition on the two half-lines forming the boundary of this sector,  Lemma {\ref{lem:2}} yields an analytic continuation of $u_1^{(N)}$ to the whole plane $\mathbb{R}^2$. Thus, in either situation, $u_1^{(N)}$ extends analytically to $\mathbb{R}^2$. By successively reversing the above reflections, we conclude that $
u_1^{(N)}, u_1^{(N-1)}, \ldots, u_1^{(0)} = u_1
$ also admit analytic continuation to the whole plane. In particular, $u_1$ is an entire solution of the Navier equation.
We can now proceed as in Case (i). Since
$$
u_1=u^{\mathrm{in}}+u_1^{\mathrm{sc}}
$$
and both $u_1$ and $u^{\mathrm{in}}$ are entire solutions of the Navier equation, it follows that the scattered field $u_1^{\mathrm{sc}}$ is an entire radiating solution. Hence, by the radiation condition and Rellich's lemma, 
$$
u_1^{\mathrm{sc}}\equiv0.
$$
Consequently,
$$
u_1=u^{\mathrm{in}}\ \text{in }D_1^c.
$$
Since $u_1$ satisfies the traction-free boundary condition on $\partial D_1$, we obtain
$
T_\nu(u^{\mathrm{in}})=0\ \text{on }\partial D_1.
$
The argument established in Case \ref{case: corner point},  then gives
$$
c_p=c_s=0,
$$
which contradicts the assumption
$$
|c_p|+|c_s|\neq0.
$$

Therefore, the assumption $D_1\neq D_2$ cannot hold, and hence
$
D_1=D_2.
$  This completes the proof.
\end{proof}

As a bi-product the uniqueness proof, we can conclude that the total field or the scattered field must be 'singular' at each corner point of the obstacle.
\begin{corollary}\label{Co-3.5}
Let $u=u^{in}+u^{sc}$ be the unique total field of the elastic scattering problem \eqref{eq:7}-\eqref{eq:9}. Then $u$ and $u^{sc}$ cannot be real-analytic around each corner of the traction-free obstacle $D$. 	
\end{corollary}	
Corollary \ref{Co-3.5} can be proved by contradiction  following the lines in the proof of Theorem~\ref{thm:main_result}. We omit its proof for brevity.	 

\section{Appendix}\label{Appen}
The following result is needed in our continuation and extension arguments. Although it appears to be well known, we present it here to make the exposition self-contained. 
\begin{proposition}\label{pro}
Suppose that $\mathcal{N}_\omega u := \mathcal{L}u+\omega^2 u=0$ in $\mathbb{R}^2$ and $T_\nu (u)=0$ on $\gamma$, where $\gamma \subset 	\mathbb{R}^2$ is a finite line segment. Denote by R the reflection with respect to some straight line and define $v:=R_{\Lambda}\circ u\circ R_{\Lambda}$ in $\mathbb{R}^2$. Then v satisfies the same Navier equation $\mathcal{N}_\omega u=0$ in $\mathbb{R}^2$ and the boundary condition $T_\nu (v)=0$ on $R(\gamma)$.
\end{proposition}	
\begin{proof}
We first show that $v$ solves the Navier equation in the whole plane. For every $x=(x_1,x_2)\in\mathbb{R}^2$, a direct differentiation gives
\[
(\nabla\cdot v)(x)=(\nabla\cdot u)(Rx),\qquad
(\Delta v)(x)=R\,(\Delta u)(Rx),\qquad
\big(\nabla(\nabla\cdot v)\big)(x)=R\,\nabla(\nabla\cdot u)(Rx).
\]
Since $\mathcal{L}=\mu\,\Delta+(\lambda+\mu)\,\nabla(\nabla\cdot)$, it follows that
\[
(\mathcal{L}v)(x)=\mu\,(\Delta v)(x)+(\lambda+\mu)\,\big(\nabla(\nabla\cdot v)\big)(x)
=R\,\big[\mu\,\Delta u+(\lambda+\mu)\,\nabla(\nabla\cdot u)\big](Rx)=R\,(\mathcal{L}u)(Rx).
\]
Consequently,
\[
(\mathcal{N}_\omega v)(x)=(\mathcal{L}v)(x)+\omega^2v(x)
=R\,(\mathcal{L}u)(Rx)+\omega^2R\,u(Rx)=R\,(\mathcal{N}_\omega u)(Rx)=0
\]
for all $x\in\mathbb{R}^2$, where the last equality uses that $\mathcal{N}_\omega u=0$ in $\mathbb{R}^2$.

Next we show that the traction-free condition is preserved. Recall that the traction is given by
\[
T_\nu(u)=\sigma(u)\,\nu,\qquad
\sigma(u):=\lambda\,(\nabla\cdot u)\,I+2\mu\,\varepsilon(u),\qquad
\varepsilon(u):=\tfrac12\big(\nabla u+(\nabla u)^{\!\top}\big),
\]
where $\sigma(u)$ and $\varepsilon(u)$ denote the stress and strain tensors; this coincides with the expression for $T_\nu$ introduced in Section~\ref{mathematical formulation}. From $v=R\,u\circ R$ one obtains, for every $x\in\mathbb{R}^2$,
\[
\nabla v(x)=R\,(\nabla u)(Rx)\,R^{\!\top},\qquad
\varepsilon(v)(x)=R\,\varepsilon(u)(Rx)\,R^{\!\top},\qquad
(\nabla\cdot v)(x)=(\nabla\cdot u)(Rx).
\]
Hence, since $R\,I\,R^{\!\top}=I$,
\[
\sigma(v)(x)=\lambda\,(\nabla\cdot v)(x)\,I+2\mu\,\varepsilon(v)(x)=R\,\sigma(u)(Rx)\,R^{\!\top}.
\]
Now fix $x\in\gamma$ and let $\nu=\nu(x)$ be a unit normal to $\gamma$ at $x$. Setting $y:=Rx\in R(\gamma)$, we have $x=Ry$, and $\nu_y:=R\nu$ is a unit normal to $R(\gamma)$ at $y$. Therefore
\[
T_{\nu_y}(v)(y)=\sigma(v)(y)\,\nu_y=R\,\sigma(u)(x)\,R^{\!\top}R\,\nu=R\,\sigma(u)(x)\,\nu=R\,T_\nu(u)(x)=0,
\]
because $T_\nu(u)=0$ on $\gamma$. As $x$ runs over $\gamma$, the point $y=Rx$ runs over $R(\gamma)$. Hence we arrive at $T_\nu(v)=0$ on $R(\gamma)$. This completes the proof.
\end{proof}	

\begin{center}
{\Large \textbf{Acknowledgement}}
\end{center}
The work of G.~Hu is supported by the National Natural Science
Foundation of China (No.~12425112), the Fundamental Research
Funds for the Central Universities in China  (No.~63263073) and
a Key Program (No.~25JCZDJC00970) of Natural Science Foundation of Tianjin, China.
V.~Rana work is supported by a Senior Research Fellowship from the
University Grants Commission (UGC), Government of India.
The work of M.~Vashisth is supported by ARG-MATRICS grant ANRF/ARGM/2025/002368/MTR,
from ANRF, Government of India. This work is partially supported by the FIST program of the Department of Science and Technology, Government of India, Reference No. SR/FST/MS-I/2018/22(C).

\end{document}